\documentclass[11pt]{amsart}
\usepackage{amssymb,amsmath,mathtools}
\usepackage{enumerate}
\usepackage{xcolor}
\usepackage{caption}
\usepackage[centering]{geometry}

\usepackage{graphicx}
\usepackage{color}
\usepackage{amsmath}
\usepackage{amssymb}
\usepackage{amscd}
\usepackage{framed}
\usepackage{bbm}
\usepackage{tcolorbox}

\usepackage{amsthm}

\newcommand{\R}{\mathbb{R}}
\newcommand{\inr}[1]{\left\langle #1 \right\rangle}

\newcommand{\PP}{\mathbb{P}}

\newtheorem{theorem}{Theorem}[section]
\newtheorem{proposition}[theorem]{Proposition}
\newtheorem{lemma}[theorem]{Lemma}
\newtheorem{corollary}[theorem]{Corollary}
\theoremstyle{definition}
\newtheorem{definition}[theorem]{Definition}
\newtheorem{assumption}[theorem]{Assumption}
\newtheorem{example}[theorem]{Example}
\newtheorem{remark}[theorem]{Remark}

\newcommand{\E}{\mathbb{E}}
\renewcommand{\P}{\mathbb{P}}

\renewcommand{\tilde}{\widetilde}

\numberwithin{equation}{section}

\newcommand{\eps}{\varepsilon}

\begin{document}

\title{Do we really need the Rademacher complexities?}

\author{Daniel Bartl}
\address{Department of Mathematics, Department of Statistics and Data Science, National University of Singapore}
\email{bartld@nus.edu.sg}
\author{Shahar Mendelson}
\address{Department of Mathematics,  Texas A\&M University}
\email{shahar.mendelson@gmail.com}
\date{\today}

\begin{abstract}
We study the fundamental problem of learning with respect to the squared loss in a convex class. 
In the subgaussian setting, the sharpest known sample complexity bounds are characterised via fixed points associated with the limiting gaussian process. In more general settings (for example, heavier-tailed scenarios) the state-of-the-art bounds rely on localised Rademacher complexities, which are often difficult to control and can be considerably larger than their limiting gaussian counterparts.

We show that if one has access to coarse information on the covariance structure of the underlying class, it is possible to derive an upper bound on the sample complexity that  matches the state-of-the-art rates  in the gaussian setting---even in heavy-tailed regimes.
As a consequence, under rather minimal assumptions, learning problems that share the same $L_2$-structure---even those involving heavy-tailed distributions---admit the same upper bound on the sample complexity as if the problem were fully gaussian. 
This constitutes the first universality result for general convex learning problems.

The proof is based on a novel learning procedure, and its performance is studied by combining optimal mean estimation techniques for real-valued random variables with Talagrand's generic chaining method.
\end{abstract}

\setcounter{equation}{0}
\setcounter{tocdepth}{1}

\maketitle
\section{Introduction}
On the face of it, the question of whether the Rademacher complexities are truly needed seems preposterous: the Rademacher complexities have had a central role in machine learning and statistical learning theory for more than 25 years; they appear in almost every state-of-the-art sample complexity estimate, and their appearance at each instance is simply natural.

Intuitively, the Rademacher complexities capture the `empirical oscillations'  of the class of functions participating in a statistical learning problem.
The idea is that if those oscillations---corresponding to the `size' of the class---are too large, no learning procedure is able to select the right function from the class.
And unfortunately, because the Rademacher complexities capture what is a rather subtle local behaviour within the class, these parameters are often difficult to handle.

At the same time, the uniform central limit theorem shows that when the sample size $N$ tends to infinity, the empirical oscillations defining the Rademacher complexities converge in distribution to a gaussian process. 
This limiting process is simply the gaussian process indexed by the class of functions considered in the learning problem.
But there were no indications that the limiting gaussian process had anything to do with the sample complexity of the problem (beyond special situations in which the empirical oscillations were trivially dominated by that gaussian process). 
That is completely natural: unlike the limiting gaussian process, the sample complexity is a `finite $N$ phenomenon'.
 In fact, the sharpest known sample complexity estimates---given in terms of the Rademacher complexities---, tend to be much larger than their limiting gaussian versions, especially in learning problems that involve heavy-tailed random variables.

The main result of this article is therefore totally unexpected: that the limiting gaussian process determines the error of an optimal learning procedure.

\begin{tcolorbox}
We design a learning procedure that exhibits finite-sample errors  that depend purely on the behaviour of the underlying  limiting gaussian process, even in heavy-tailed problems.
The (more complicated and often much larger) Rademacher complexities are not needed.
\end{tcolorbox}

To explain what this means and how it improves the current state-of-the-art, we must first describe the setting of the learning problems we focus on.

\subsection{The classical learning problem}

In what follows we consider the classical learning scenario: learning with respect to the squared loss, and for an underlying class of functions that is  convex and closed. The study of such learning problems has been of central importance in statistical learning theory since the very early days of the area---almost 60 years ago.
For more information we refer e.g.\ to the books \cite{AnBa99,BuvdG11,DGL96,Kolt08,Mas06,Ts09,vdG00,VaCh74a} and references therein.

A learning problem consists of a triplet: a probability space $(\Omega,\mu)$,  a  closed convex  class of functions $F \subset L_2(\mu)$ (which contains a countable subset that can be used to approximate every function in the class pointwise and in $L_2(\mu)$) and a target random variable $Y$; if $X$ is distributed according to $\mu$, the triplet is denoted by $(F,X,Y)$.  The crucial point is that the learner does not have access to all this information. 
Usually, all that the learner knows are the identities of $\Omega$ and of the class of functions $F$, but not the underlying distribution of $X$ or the target random variable $Y$. 
Still, the goal is to find a good approximation of $Y$ in $F$, and the information that can be used to achieve that goal (other than the identity of $F$), is a sample $(X_i,Y_i)_{i=1}^N$, selected independently according to the joint distribution $(X,Y)$.

Naturally, `good approximation' can be understood in a variety of ways, and is dictated by the choice of a \emph{loss functional}. Here, we focus on what is arguably the most natural choice of a loss --- the \emph{squared loss}: the loss caused by `guessing' $f(x)$ instead of $y$ is $\ell(f(x)-y)=(f(x)-y)^2$.

For every $f \in F$, the \emph{risk} of the function $f$ is $\E(f(X)-Y)^2$, and thanks to the convexity of $F$ and the fact that it is closed,  a risk minimiser in $F$ exists and is unique.
Set $f^*={\rm argmin}_{f \in F} \E (f(X)-Y)^2$ and let
$$
{\mathcal L}_f(X,Y) =(f(X)-Y)^2-(f^*(X)-Y)^2
$$
be the \emph{excess loss} associated with $f$.
It is standard to verify that thanks to the convexity of $F$,
\begin{equation} \label{eq:convexity condition}
\|f-f^*\|_{L_2}^2 \leq \E {\mathcal L}_f;
\end{equation}
this so-called  \emph{convexity condition} plays an important role in what follows.

\begin{remark}
Keep in mind that the learner has no information on ${\mathcal L}_f$---even on the given sample: the value ${\mathcal L}_f(X_i,Y_i)$ depends on $f^\ast(X_i)$,  but obviously $f^*$ is not known to the learner.
\end{remark}

The two main problems in this setup are \emph{estimation} and \emph{prediction}: how to use the sample to produce a function $\widehat{f} \in F$ for which $\|\widehat{f}-f^*\|_{L_2}^2$ is `small' (estimation), or that the conditional expectation $\E {\mathcal L}_{\widehat{f}}$ is small (prediction). In both cases, the key question is the best accuracy/confidence tradeoff  that one can have: given a sample size $N$, to find the best accuracy $\eps$ and confidence parameter $\delta$ for which, with probability at least $1-\delta$,
\begin{equation} \label{eq:tradeoff}
\|\widehat{f}-f^*\|_{L_2}^2 \leq \eps \ \ {\rm and/or} \ \ \E {\mathcal L}_{\widehat{f}} \leq \eps.
\end{equation}
Naturally, identifying the \emph{learning procedure} $\widehat{f}$ that achieves that tradeoff is essential.

An alternative and equivalent formulation of the same problem is to obtain the optimal \emph{sample complexity} estimate: that is, given the wanted accuracy and confidence levels, find the smallest sample size  $N$ for which  \eqref{eq:tradeoff} holds.

\vspace{0.5em}
Identifying the optimal tradeoff/sample complexity in this setup has been the topic of hundreds of articles  and dozens of books (see for example  the references mentioned previously and \cite{BaBoMe05,GyKoKrWa02,lugosi2019risk,MenACM,mendelson2017local}).
And what was believed to be an (almost) complete solution in the setup studied here was established in \cite{lugosi2019risk} (and then extended to the non-convex case in \cite{mendelson2019unrestricted}): it was shown that there is a learning procedure for which the tradeoff is governed by the behaviour of the empirical oscillations,  leading to the \emph{Rademacher complexities} of the triplet.

\subsection{The Rademacher complexities}

Over the years, the (localised) Rademacher averages have been the key ingredient in the analysis of learning problems.
Among their roles, they are used to define the Rademacher complexities of a triplet $(F,X,Y)$.

Since their introduction in the early 2000s, the Rademacher complexities have captured the state-of-the-art sample complexity estimates. For example, triplets consisting of uniformly bounded functions (see, e.g.,  \cite{BaBoMe05,Kolt08}),  subgaussian triplets (see, e.g., \cite{lecue2013learning}),  and heavy-tailed problems  (see, e.g., \cite{lecue2020robust,lugosi2019risk,oliveira2025trimmed}). 
Their role was so central, that it was conjectured that they captured the optimal sample complexity of a learning problem under minimal assumptions.  

As noted previously, the Rademacher complexities are fixed points that characterise the local oscillations of natural empirical processes associated with the learning problem. More accurately,  let $(\eps_i)_{i=1}^N$ be independent, symmetric, $\{-1,1\}$-valued random variables that are independent of $(X_i,Y_i)_{i=1}^N$.
For $r>0$, set
\[\Phi_N(r) = \E \sup_{u \in (F-F) \cap rD} \left| \frac{1}{\sqrt{N}} \sum_{i=1}^N \eps_i u(X_i) \right| ,\]
where $D$ is the $L_2$ unit ball and $F-F=\{ f-h : f,h\in F\}$.

In a similar fashion, let $\xi= f^\ast(X)-Y$, set $\xi_i=f^\ast(X_i)-Y_i$, and define
\[\Phi_{N,\xi} (r) = \E \sup_{u \in (F-F) \cap rD} \left| \frac{1}{\sqrt{N}} \sum_{i=1}^N \eps_i \xi_i u(X_i)  \right| . \]

\begin{definition}
\label{def:rademacher}
For $\kappa>0$ let
\begin{equation} \label{eq:emp-r-Q}
r_{\mathcal{Q}}^{\rm rad}(\kappa)
= \inf\left\{ r>0: \Phi_N(r)  \leq \kappa \sqrt{N} r\right\}.
\end{equation}
\end{definition}

Clearly, the set $F-F$ is star-shaped around $0$: if $u \in F-F$ then for any $\lambda \in [0,1]$, $\lambda u \in F-F$. As a result, it is straightforward to verify that $\frac{1}{r}\Phi_N(r)$ is decreasing for $r>0$---because the sets $(\frac{1}{r}(F-F) )  \cap D$ become `richer' as $r$ decreases.
Moreover, $r_{\mathcal{Q}}^{\rm rad}(\kappa) $ is a solution to the equation $\Phi_N(r)  = \kappa \sqrt{N} r$; when $r \geq r_{\mathcal{Q}}^{\rm rad}(\kappa) $ we have that $\Phi_N(r)  \leq \kappa \sqrt{N} r$; and when $r \leq r_{\mathcal{Q}}^{\rm rad}(\kappa) $ the reverse inequality holds.

It was shown in \cite{mendelson2017local} that for a suitable choice of the constant $\kappa$, $r_{\mathcal{Q}}^{\rm rad}(\kappa) $ captures the `intrinsic complexity' of the class, and in particular, is the dominating factor in the sample complexity of `low noise' problems. The noise level $\sigma$ is defined by
\[\sigma^2=\E(f^*(X)-Y)^2=\|f^\ast-Y\|_{L_2}^2,\]
and the problem is considered to have \emph{low noise} if $\sigma \leq  r_{\mathcal{Q}}^{\rm rad}(\kappa) $ (for a more detailed explanation, see \cite{BaBoMe05} and \cite{mendelson2017local}).

As it happens, when the noise level exceeds $r_{\mathcal{Q}}^{\rm rad}(\kappa)$, the dominating factor is a different fixed point.

\begin{definition}
Set $\xi=f^*(X)-Y$, and for $\kappa>0$ define
\begin{equation} \label{eq:emp-r-M}
r_{\mathcal{M}}^{\rm rad}(\kappa)  = \inf\left\{ r>0: \Phi_{N,\xi}(r) \leq \kappa \sqrt{N} r^2\right\}.
\end{equation}
\end{definition}

Once again, using the star-shape property of $F-F$, it is evident that $r_{\mathcal{M}}^{\rm rad}(\kappa)$ is given by a fixed point condition; that for $r \geq r_{\mathcal{M}}^{\rm rad}(\kappa)$ we have that $\Phi_{N,\xi}(r) \leq \kappa \sqrt{N} r^2$; and when $r \leq r_{\mathcal{M}}^{\rm rad}(\kappa)$ the reverse inequality holds.

\begin{tcolorbox}
The fixed points $r_{\mathcal{Q}}^{\rm rad}(\kappa)$ and $r_{\mathcal{M}}^{\rm rad}(\kappa)$ are the \emph{Rademacher complexities} associated with the triplet $(F,X,Y)$.
\end{tcolorbox}

To give some intuition, note that the Rademacher average $\Phi_N$ measures how well a typical random set $\{ (h(X_i))_{i=1}^N : h \in (F-F) \cap rD\}$ is correlated with pure external noise in the form of the Bernoulli random vector $(\varepsilon_i)_{i=1}^N$; similarly $\Phi_{N,\xi}$ measures the correlation of the same typical random set with the given noise vector $(\xi_i)_{i=1}^N$.
$\Phi_N$ and $\Phi_{N,\xi}$ capture a natural notion of complexity of  the class $F$:
 they serve as an indication of ``how much oscillation'' can one expect locally in $F$. 
As it happens, large oscillations are the reason why the learner is misled by the given data:
they cause functions that are relatively far from $f^\ast$ to masquerade as $f^\ast$.
The Rademacher complexities capture the distances at which these local oscillations become too big and can cause such mistakes. 

The sharpest bound on the accuracy/confidence tradeoff (and on the sample complexity) in our setting was established in \cite{lugosi2019risk}, and is based on the Rademacher complexities.

The other two complexity parameters that were used in \cite{lugosi2019risk} are fixed points given in terms of packing numbers of the class $F$ with respect to the $L_2$ norm. Denote by ${\mathcal M}(A,rD)$ the cardinality of a maximal $r$-separated subset of $A$ with respect to the $L_2$ norm.

\begin{definition}
For $\kappa,\gamma>0$ let
\begin{equation} \label{eq:covering-fixed-point-1}
\lambda_{\mathcal{Q}}(\kappa,\gamma) =  \inf\{ r>0: \log {\mathcal M}((F-F) \cap rD,\gamma r D) \leq \kappa^2 N\},
\end{equation}
and
\begin{equation} \label{eq:covering-fixed-point-2}
\lambda_{\mathcal{M}}(\kappa,\gamma) =  \inf\{ r>0: \log {\mathcal M}((F-F) \cap rD,\gamma r D) \leq \kappa^2 N r^2\}.
\end{equation}
\end{definition}

To formulate the  estimate from \cite{lugosi2019risk}  we require the following:

\begin{assumption} \hfill
\label{ass:LugMen}
\begin{itemize}
\item For every $f,h \in F \cup \{0\}$, $\|f-h\|_{L_4} \leq L \|f-h\|_{L_2}$;
\item  for every $f \in F$, $\|f-Y\|_{L_4} \leq L\|f-Y\|_{L_2}$;
\item $\|f^*-Y\|_{L_2} \leq \overline{\sigma}$ for some known  constant $\overline{\sigma}>0$.
\end{itemize}
\end{assumption}

The final condition in Assumption \ref{ass:LugMen} means that the learner has some a priori information on the noise level of the problem and the resulting accuracy/confidence tradeoff depends on $\overline{\sigma}$.

The first two components of Assumption \ref{ass:LugMen} are that the class $F \cup \{0\} $ satisfies $L_4-L_2$ norm equivalence with constant $L$, as does the class $\{f-Y : f \in F\}$; neither  assumption is really restrictive.

\begin{example}
There are many scenarios in which  $X \in \R^d$  is a centred random vector and satisfies that for every $z \in \R^d$,  $\|\inr{X,z}\|_{L_4} \leq L \|\inr{X,z}\|_{L_2}$; therefore, every class of linear functionals of the random vector  satisfies the first part of Assumption \ref{ass:LugMen}.
Here we present two such examples.
Let $x$ be a mean-zero, variance $1$ random variable, for which $\|x\|_{L_4} \leq  L$.
The random vector $X=(x_1,...,x_d)$ that has independent copies of $x$ as coordinates satisfies  $L_4-L_2$ norm equivalence with constant $3^{1/4}L$.

At the other, light-tailed end of the spectrum,  one can show that for any log-concave random vector,  every $p \geq 2$ and every $z \in \R^d$,
\begin{align}
\label{eq:psi1}
\|\inr{X,z}\|_{L_p} \leq Lp\|\inr{X,z}\|_{L_2}
\end{align}
 for $L$ that is an absolute constant.
 Indeed, \eqref{eq:psi1}  follows from Borell's inequality (see, e.g.\ \cite{artstein2015asymptotic}).
 Recalling that for a random variable $x$ and $\alpha \in [1,2]$,
$$
\|x\|_{\psi_\alpha} = \sup_{p \geq 2} \frac{\|x\|_{L_p}}{p^{1/\alpha}},
$$
\eqref{eq:psi1} means that  a log-concave random vector satisfies a $\psi_1-L_2$ norm equivalence for an absolute constant $L$.
\end{example}

The following theorem is the main result in \cite{lugosi2019risk} and is the current `gold standard'  estimate on learning in convex classes with respect to the squared loss.
We also refer to \cite{lecue2020robust,oliveira2025trimmed} for recent results, which include relaxations of Assumption \ref{ass:LugMen}, adversarial contamination, and regularized learning problems.

\begin{theorem} \label{thm:main-LugMen}
There are constants $c,c_0,c_1$ and $c_2$ that depend only on $L$ for which the following holds.
Let
\begin{equation} \label{eq:fixed-points-LugMen}
r_{\rm rad}^*=\max\left\{\lambda_{\mathcal{Q}}(c_1,c_2),\lambda_{\mathcal{M}}(c_1/\overline{\sigma},c_2),r_{\mathcal{Q}}^{\rm rad}(c_1),r_{\mathcal{M}}^{\rm rad}(c_1) \right\}~,
\end{equation}
and fix $r \geq 2r_{\rm rad}^*$. There exists a procedure that, based on the data
${\mathcal D}_N=(X_i,Y_i)_{i=1}^N$ and the values of $L$, $\overline{\sigma}$ and $r$, selects a
function $\widehat{f} \in F$ which satisfies that with probability at least
$$
1-\exp\left(-c_0N \min\left\{1,\frac{r^2}{\overline{\sigma}^2}\right\}\right),
$$

$$
\|\widehat{f}-f^*\|_{L_2} \leq c r \ \ {\rm  and } \ \
\E \bigl({\mathcal L}_{\widehat{f}} \,| \,{\mathcal D}_N \bigr) \leq cr^2.
$$
\end{theorem}

An alternative formulation of Theorem \ref{thm:main-LugMen} (that can be found in \cite{mendelson2019unrestricted}) is  the corresponding sample complexity estimate; it uses the fact that $r_{\rm rad}^*$ decreases in $N$.

\begin{corollary} \label{cor:main-LugMen}
Given $\eps$ and $\delta$, let $N_0(\eps)$ be the smallest integer for which $r_{\rm rad}^* \leq c_3 \sqrt{\eps}$. If
$$
N \geq N_0(\eps) + c_4\left(\frac{\overline{\sigma}^2}{\eps}+1\right) \log\left(\frac{2}{\delta}\right),
$$
then with probability at least $1-\delta$,
$$
\|\widehat{f}-f^*\|_{L_2} \leq  \sqrt{\varepsilon} \ \ {\rm  and } \ \
\E \bigl({\mathcal L}_{\widehat{f}} \,| \,{\mathcal D}_N \bigr) \leq \varepsilon.
$$
\end{corollary}

\begin{remark}
It is important to stress that the most natural learning procedure, \emph{empirical risk minimisation (ERM)}---that selects a function in $F$ that minimizes the empirical risk $\frac{1}{N}\sum_{i=1}^N (f(X_i)-Y_i)^2$---, is far from optimal.
In heavy-tailed situations its performance does not come close to that of the \emph{tournament procedure} introduced in the proof of Theorem \ref{thm:main-LugMen}.
\end{remark}

\begin{example}[Linear regression in $\R^d$]
 \label{sec:reg-1}
Let $T \subset \R^d$ be convex (and closed) and set $F_T=\{\inr{\cdot,t} : t \in T\}$ to be the class of linear functionals associated with $T$.
Let $X$ be a random vector that satisfies $L_4-L_2$ norm equivalence with constant $L$, and  set $w$ to be a centred random variable that is independent of $X$ and for which $\|w\|_{L_4} \leq L \|w\|_{L_2}$. 
Let $z_0 \in \R^d$ and put $Y=\inr{X,z_0}+w$.

It is standard to verify that the triplet $(F,X,Y)$ satisfies the first two components in Assumption \ref{ass:LugMen} with a constant that depends only on $L$. 
As a result, Theorem \ref{thm:main-LugMen} holds, and the Rademacher complexities are the fixed points  associated with the functions $\Phi_{N}(r)$ and $\Phi_{N,\xi}(r)$.

Thanks to the linear structure this problem has, the Rademacher complexities have a (seemingly) simpler description. Note that $T-T$ is a convex, centrally symmetric set, and therefore so is $(T-T) \cap rD$.
As such, $(T-T) \cap rD$ is the  unit ball of a norm on $\R^d$ (or on a subspace of $\R^d$), and denote that norm by $\| \cdot \|_{K_r}$. Thus, the Rademacher complexities are determined by the behaviour of  the functions
\begin{equation} \label{eq:phi-random-vector}
 \Phi_{N}(r) = \E \left\|\frac{1}{\sqrt{N}} \sum_{i=1}^N \eps_i X_i \right\|_{K_r}  \ \ {\rm and} \ \
 \Phi_{N,\xi}(r) = \E\left\|\frac{1}{\sqrt{N}} \sum_{i=1}^N \eps_i \xi_i X_i \right\|_{K_r}.
\end{equation}
Unfortunately, controlling the expected norms of the two random vectors in terms of the geometry of the set $K_r$ is  highly nontrivial and far from understood---other than in rather special situations.
These are far more complex objects than, say, $\E \|G\|_{K_r}$, where $G$ is the gaussian random vector in $\R^d$ that has the same covariance as $X$.
Indeed, while there are sophisticated methods that lead to a complete characterisation of $\E \|G\|_{K_r}$ purely in terms of the geometry of the set $K_r$ (see, for example \cite{talagrand2022upper}), there are no known characterisations for \eqref{eq:phi-random-vector}.
At the same time, it follows from the central limit theorem  that $\Phi_{N}(r) \to \E \|G\|_{K_r}$ as $N\to\infty$.
\end{example}

\subsection{If not the Rademacher complexities, then what?}

The one obvious, but at the same time highly unlikely candidate that  one might consider as a potential replacement for the Rademacher complexities are the fixed points associated with the limiting gaussian process.
In other words, instead of using parameters that capture the  oscillations of empirical/multiplier processes, one may consider the behaviour of the oscillations of the gaussian process indexed by $F$. Of course, at this point there is no hint that these oscillations have anything to do with the learning problem, and the idea of considering them as an option is speculative at best.

\begin{definition}
Let $\sigma^2=\inf_{f \in F} \E (f(X)-Y)^2$, set $\{G_h : h \in F-F\}$ to be the gaussian process indexed by $F-F$  whose covariance is endowed by $L_2(\mu)$, and put
\begin{equation} \label{eq:gaussian-fixed}
\begin{split}
r_{\mathcal{Q}}(\kappa) &= \inf\left\{ r>0 : \E \sup_{h \in (F-F) \cap rD} G_h \leq \kappa \sqrt{N} r \right\},\\
r_{\mathcal{M}}(\kappa) &= \inf\left\{ r>0 : \E \sup_{h \in (F-F) \cap rD} G_h \leq \kappa \sqrt{N} \frac{r^2}{\sigma}\right\}.
\end{split}
\end{equation}
\end{definition}

Using that $F-F$ is star-shaped around $0$, it is evident that if $r \geq r_{\mathcal{Q}}(\kappa)$ then
\[
\E \sup_{h \in (F-F) \cap rD} G_h \leq \kappa \sqrt{N} r,
\]
 and if $r \leq r_{\mathcal{Q}}(\kappa)$ then the reverse inequality holds. 
 The analogous fact is true for $r_{\mathcal{M}}(\kappa)$. 

\vspace{0.5em}
With the two new parameters, the notions of `low noise' and `high noise' have to be adjusted. 
The meaning of low noise is that $\sigma \leq r_{\mathcal{Q}}(\kappa)$, while high noise means that $\sigma \geq r_{\mathcal{Q}}(\kappa)$---which implies that $\sigma \geq r_{\mathcal{M}}(\kappa)$ as well. Indeed, if $\sigma \geq r_{\mathcal{Q}}(\kappa)$ then
\[
\E \sup_{h \in (F-F) \cap \sigma D} G_h \leq \kappa \sqrt{N} \sigma = \kappa \sqrt{N} \frac{\sigma^2}{\sigma},
\]
and therefore $r=\sigma$ is an `eligible candidate' in the definition of $r_{\mathcal{M}}(\kappa)$.

\vspace{0.5em}

A conjecture that for a well-chosen $\kappa$,
$$
r^\ast(\kappa) = \max\{r_{\mathcal{Q}}(\kappa), r_{\mathcal{M}}(\kappa)\}
$$
might bound the optimal tradeoff / sample complexity at the constant probability level seems wildly optimistic. Firstly,  the empirical oscillation functions $\Phi_N(r)$ and $\Phi_{N,\xi}(r)$ tend to be much bigger than their  limiting  gaussian counterpart---especially in heavy-tailed situation (see Section \ref{sec:lin.reg.B1d}  for an example).
Secondly, a bound that is based solely on the behaviour of the limiting gaussian process does not `see' all the fluctuations caused by the random sample.
Such a bound would be the same for all distributions that endow the same $L_2$ structure, and would coincide with the light-tailed estimate even if the problem is heavy-tailed.

Taking all that into account, such a conjecture seems to be too good to be true.

Despite that, our main result, presented in the next section, is that essentially, \linebreak $\max\{r_{\mathcal{Q}}(\kappa), r_{\mathcal{M}}(\kappa)\}$ is an upper bound on the optimal tradeoff  even in heavy-tailed problems, once we assume that the learner is given a little more information on the problem.

\subsection{The main result}

The `little more information' that one requires is an \emph{isomorphic distance oracle} on $F \cup \{0\}$.

\begin{assumption}
\label{ass:distance.oracle}
The learner has access to a symmetric\footnote{That is, $d(f,h)=d(h,f)$.} function $d$ that satisfies, for every $f,h\in F\cup\{0\}$,
\[\frac{1}{\eta} \|f-h\|_{L_2} \leq d(f,h)\leq \eta \|f-h\|_{L_2},\]
for some fixed $\eta\geq 1$.
\end{assumption}

Assumption \ref{ass:distance.oracle} implies that the learner has access to \emph{crude} information on $L_2$ distances between functions in $F$, as well as their $L_2$ norms.
Crucially, the learner does not have more information on $Y$ or on functions that depend on $f-Y$, nor is there a need for accurate information on $X$.
A situation in which Assumption \ref{ass:distance.oracle} is natural is presented in Example \ref{ex:lin.reg.known.Sigma}.

\begin{remark}
As we explain in what follows, thanks to the distance oracle and to the \emph{majorizing measures theorem}, the learner can identify $\E \sup_{h \in (F-F) \cap rD} G_h$ for every $r>0$ (up to multiplicative constants that depend on $\eta$). As a result, the learner knows the value of $r_{\mathcal{Q}}(\kappa)$---again, up to multiplicative constants in $\eta$. The situation regarding $r_{\mathcal{M}}(\kappa)$ is more subtle, as its definition is based on the unknown noise level $\sigma$.  
While one could assume, as in \cite{lugosi2019risk}, that the noise level satisfies $\sigma \leq \overline{\sigma} $ for some known constant $\overline{\sigma}$, in the present setting we construct a data-dependent preliminary estimate of the noise level.
This estimate will serve as an initial step in our learning procedure.
\end{remark}

We also assume that the first two parts of Assumption \ref{ass:LugMen} hold: that $\|f-h\|_{L_4}\leq L \|f-h\|_{L_2}$ for every $f,h\in F\cup\{0\}$, and that $\|f-Y\|_{L_4}\leq L \|f-Y\|_{L_2}$ for every $f\in F$.

\vspace{0.5em}

The first ingredient we need is an `isomorphic' estimate on the noise level $\sigma^2=\E (f^\ast(X)-Y)^2$.

\begin{lemma}
\label{lem:noise.estimate}
	There are constants $c_1,c_2,c_3$ that depend only on $\eta$ and $L$ such that the following holds.
	If $r$ satisfies  that $\log \mathcal{M}(F,c_1rD) \leq c_2 N$, then there is a procedure that, upon receiving the data ${\mathcal D}_N=(X_i,Y_i)_{i=1}^N$, the distance oracle $d$, and the values of $L,\eta,r$, selects $\widehat\sigma$ which satisfies that with probability at least $1 - 2\exp\left( - c_3 N \right)$,
\[ \widehat{\sigma}^2 \leq 2 \max\{ \sigma^2, r^2\}, \quad\text{and if $\sigma>r$, then }  \widehat{\sigma}^2 \geq \frac{1}{2}\sigma^2 .\]
\end{lemma}

\begin{remark}
We will show in what follows that the role of the entropic condition $\log \mathcal{M}(F,c_1rD) \leq c_2 N$ is minimal---see Lemma \ref{lem:r.almost.implies.lambda}.
\end{remark}

Lemma \ref{lem:noise.estimate} gives a \emph{data-dependent} way of controlling $r^\ast$:  it is straightforward  to verify that
$$
\min \left\{1,\frac{r}{\sigma} \right\}\sim\min \left\{1,\frac{r}{\widehat{\sigma}} \right\},
$$ 
and thus,
\[
\widehat{r}^{\, \ast}(\kappa) =  \inf\left\{ r>0 : \E \sup_{h \in (F-F) \cap rD} G_h \leq \kappa \sqrt{N} r \min\left\{1, \frac{r}{ \widehat{\sigma} }\right\} \right\}
\]
is equivalent to $r^\ast=\max\{ r_{\mathcal{Q}}, r_{\mathcal{M}}\}$.

The last ingredient we need is an entropic fixed point.
Let
\begin{align}
\label{eq:def.lambda.our}
\lambda^\ast(\kappa) = \inf\left\{ r>0 : \log \mathcal{M}(F, rD)
\leq \kappa  N \min\left\{1, \frac{r^2}{\sigma^2} \right\} \right\}.
\end{align}
Again, replacing $\sigma$ by $\widehat\sigma$  in the definition of $\lambda^\ast(\kappa)$ results in an equivalent condition.
Also note that the condition that $\log \mathcal{M}(F, c_1rD)\leq c_2 N$ appearing in Lemma \ref{lem:noise.estimate} is satisfied if $r> \frac{1}{c_1} \lambda^\ast(c_2)$.

With all the ingredients set in place,  let us formulate our main result.

\begin{tcolorbox}
\begin{theorem} \label{thm:main}
There are constants $c_0,c_1,c_2$ that depend only on $L$ and $\eta$ for which the following holds.
If $r> c_0 \max\{ r^\ast(c_1), \lambda^\ast(c_1)\}$, then there exists a procedure that, based on the data ${\mathcal D}_N=(X_i,Y_i)_{i=1}^N$, the distance oracle $d$, and the values of $L,\eta$ and $r$, selects a function $\widehat{f} \in F$ which satisfies that with probability at least
\[1 - 2\exp\left( - c_2 N \min\left\{1,\frac{r^2}{\sigma^2}\right\}\right), \]

\begin{align}
\label{eq:main.thm.estimate}
\|\widehat{f} - f^\ast\|_{L_2} &\leq r \quad\text{and}\quad
\E \bigl({\mathcal L}_{\widehat{f}} \,| \,{\mathcal D}_N \bigr) \leq  r^2.
\end{align}
\end{theorem}
\end{tcolorbox}

\begin{remark}
	As the procedure selecting $\widehat{f}$  is a rather complicated function of the data, it might happen that the event \eqref{eq:main.thm.estimate} is not measurable.
	Hence, here and in what follows, statements such as  $\P(\eqref{eq:main.thm.estimate} \text{\, holds})\geq 1-\delta$ are to be understood as: ``there is a measurable set $E$ satisfying $\P(E)\geq 1-\delta$ and for every realisation in $E$, \eqref{eq:main.thm.estimate} holds''.
\end{remark}

Thanks to Theorem \ref{thm:main}, it is straightforward to derive the corresponding sample complexity bounds:

\begin{corollary} \label{cor:main}
There are constants $c_1,c_2,c_3$ that depend only on $L$ and $\eta$ for which the following holds.
Given $\eps$ and $\delta$, let $N_0(\eps)$ be the smallest integer for which 
$\max\{ r^\ast(c_1), \lambda^\ast(c_1)\}\leq c_2\sqrt{\eps}$.
If
\begin{align}
\label{eq:sample.complexity.N}
N \geq N_0(\eps) + c_3\left(\frac{\sigma^2}{\eps}+1\right) \log\left(\frac{2}{\delta}\right),
\end{align}
then with probability at least $1-\delta$,
\begin{align}
\label{eq:oose}
\|\widehat{f}-f^*\|_{L_2} \leq  \sqrt{\varepsilon} \ \ {\rm  and } \ \
\E \bigl({\mathcal L}_{\widehat{f}} \,| \,{\mathcal D}_N \bigr) \leq \varepsilon.
\end{align}
\end{corollary}

And as in the fixed point conditions before, $\sigma$ in \eqref{eq:sample.complexity.N} may be replaced with $\widehat{\sigma}$.

We describe the differences  between Theorem \ref{thm:main} and the results established in \cite{lecue2020robust,lugosi2019risk,oliveira2025trimmed} in more detail in Section \ref{sec:comparison}.

\begin{example}[Linear regression in $\R^d$, revised]
\label{ex:lin.reg.known.Sigma}
Just as in Example \ref{sec:reg-1} and using its notation, set $X\in \R^d$ to be the random vector, let $T\subset\R^d$ be a convex and compact set, put  $F=\{\inr{\cdot,t} : t\in T\}$ and consider $Y=\inr{X,z_0}+w$. As noted in Example \ref{sec:reg-1}, the triplet $(F,X,Y)$ satisfies the required norm equivalence assumptions.
	
Assumption \ref{ass:distance.oracle} is satisfied if there is some additional crude information on the covariance matrix of $X$ (denoted by $\Sigma_X$ in what follows). For example, it suffices that the learner has access to a positive semi-definite matrix $A$  for which
	\[ \eta^{-2} \Sigma_X \preceq A \preceq \eta^2 \Sigma_X,\]
	and in that case one may set
\[
d^2(u,v)=\inr{A(u-v),u-v}.
\]
There are many natural linear regression problems  where the covariance matrix $\Sigma_X$ is known, and then Assumption \ref{ass:distance.oracle} is trivially satisfied with $\eta=1$.

In Section \ref{sec:lin.reg.B1d} we show that even in this simple setting, one can construct an example for which $N_0 \sim \log(d)$  but the estimate that follows from the Rademacher complexities only yields that $N_0\gtrsim d^{1/4}$.
\end{example}

The entropy condition in Theorem \ref{thm:main} that $ r\geq \lambda^\ast(\kappa)$ should not be alarming.
It is often satisfied automatically when $r\geq r^\ast(\kappa)$, and in general, $\lambda^\ast$ is only slightly larger than $r^\ast$.
Specifically, set $\theta(r) = 1/\max\{1, \sqrt{ \log(2d_F/r) } \}$ and for every $\kappa>0$, put 
	\begin{align*}
	 \tilde{r}^{\,\ast}(\kappa) = \inf\left\{ r>0 : \E \sup_{h \in (F-F) \cap rD} G_h
	\leq \kappa \sqrt{N} \theta(r) r  \min\left\{1, \frac{r}{\sigma}\right\}  \right\}.
\end{align*}
Then we have  the following:

\begin{lemma}
\label{lem:r.almost.implies.lambda}
	For every $\kappa$ there is a constant $c=c(\kappa)$ such that $\lambda^\ast(\kappa) \leq \tilde{r}^{\,\ast}(c)$.
\end{lemma}

In particular, recalling that $r\gtrsim1/ \sqrt N$, it is evident that if $N_0$ satisfies that $r^\ast(\kappa)\leq \sqrt \varepsilon$ and 
\[N\geq c \max\{1, \log(d_F^2 N_0)\} N_0,\]
 then $\lambda^\ast(\kappa)\leq \sqrt \varepsilon$.
 
The proof of Lemma \ref{lem:r.almost.implies.lambda}  and additional observations on the relation between $\lambda^\ast$ and $r^\ast$ can be found in Appendix \ref{sec:app.entropy}.

At this point, it should be stressed that the procedure we use in the proof of Theorem \ref{thm:main} is not intuitive, nor is it feasible from a computational perspective.
Still, the fact that such a procedure even exists---and that the sample complexity of an arbitrary learning problem can be controlled using the limiting gaussian process even in heavy-tailed situations---, is a complete surprise. To put Theorem \ref{thm:main} in some perspective, we shall outline the relations between the upper bounds on the sample complexity and the state-of-the-art lower bounds in Section \ref{sec:lower}.

\vspace{0.5em}

The proof of Theorem \ref{thm:main} is based  on an unorthodox chaining argument, combined with an arbitrary optimal mean estimation procedure for real-valued random variables.
We believe that this type of argument will have far reaching implications (see for example, the results in \cite{bartl2025uniform}).
The unorthodox chaining argument is described in Section \ref{sec:chaining}.

\subsection*{Remarks and notation}

The procedures we use are functions $\Psi$ that are given the values $(X_i,Y_i)_{i=1}^N$ as data. In addition, they have access to functions in $F$ (or in $F-F=\{f-h : f,h \in F  \}$), and therefore can `see' the values $(f(X_i))_{i=1}^N$ as well.
At times, the procedures also have access to the wanted confidence parameter $\delta$ and to the distance oracle $d$.
However, the procedures never know the identities of the target function $Y$ or of the risk minimiser $f^*$.

\begin{tcolorbox}
In what follows and with a slight abuse of notation, if we write, say, for $f \in F$ and $\xi=h(X)-Y$, ``$\Psi(f,\xi)$" it means that $\Psi$ ``sees" the values $(f(X_i))_{i=1}^N$, $(h(X_i))_{i=1}^N$, and $(h(X_i)-Y_i)_{i=1}^N$.
It also knows the identities of $f$ and $h$, but not that of $Y$ or of $f^*$.
All that is possible given the data the learner has access to.
\end{tcolorbox}

Throughout this article, $c, c_0, c_1$ denote strictly positive absolute constants that may change their value in each appearance. If a constant $c$ depends  on  a parameter  $\alpha$, that  is  denoted  by  $c = c(\alpha)$.
We  write  $\alpha \lesssim \beta$  if  there is an absolute constant $c$ for which  $\alpha\leq c\beta$ and $\alpha\sim\beta$ if both  $\alpha\lesssim \beta$ and  $\beta\lesssim \alpha$.

With another minor abuse of notation we write $f$ instead of $f(X)$ and $f-Y$ instead of $f(X)-Y$. Also, we may assume without loss of generality that $0 \in F$. If not, fix an arbitrary $f_0 \in F$ and consider the shifted class $F^\prime = \{f -f_0 : f \in F\}$ and the shifted target $Y-f_0$; the given data will now be $(X_i,Y_i-f_0(X_i))_{i=1}^N$. Clearly, $F-F=F^\prime-F^\prime$, and $\sup_{f^\prime \in F^\prime} \|f^\prime\|_{L_2} \sim {\rm diam}(F,L_2)$.

\section{Proof of Theorem \ref{thm:main}}

The learning procedure used in the proof of Theorem \ref{thm:main} has  several components that are described here. 
Each one of the components relies on the theorem's assumptions.

\vspace{0.5em}
The first part of the procedure is a pre-processing step that is not random. It will be presented in Section \ref{sec:chaining.admissible.seq}.

The two main components of the procedure are a \emph{crude risk oracle}, denoted by $\Psi_{\mathcal{C}}$, and a \emph{fine risk oracle}, denoted by $\Psi_\mathcal{L}$. Let $(X_i,Y_i)_{i=1}^{2N}$ be the given sample and set $r \geq c_0 \max\{r^\ast(c_1), \lambda^\ast(c_1)\}$ for constants $c_0$ and $c_1$ that depend only on $L$ and $\eta$ and are specified in what follows.

The crude risk oracle receives as input the first half of the sample,  $(X_i,Y_i)_{i=1}^N$ and the parameter $r$, and for every $f \in F$ returns $\Psi_{\mathcal{C}}(f)$. Set
$$
\widehat{\sigma}^2=\inf_{f \in F} \Psi_{\mathcal{C}}(f), \ \  \ \sigma_\ast=\max\{\widehat{\sigma},r\},
$$
and define
$$
\widehat{V}=\left\{ f \in F: \Psi_{\mathcal{C}}(f) \leq 4\sigma_\ast^2 \right\}.
$$
Creating the set $\widehat{V}$ is the crude oracle's main task.
The idea behind  $\widehat{V}$ is to exclude functions in $F$ that have an unreasonably large risk, while at the same time ensuring that $f^* \in \widehat{V}$.

The second component---the fine risk oracle, receives as input the second half of the sample $(X_i,Y_i)_{i=N+1}^{2N}$, the set $\widehat{V}$, and the values $r$ and $\sigma_\ast$. 
For every ordered pair $(f,h) \in \widehat{V} \times \widehat{V}$ it returns $\Psi_\mathcal{L}(f,h)$. 
Based on those values, a tournament is played between functions in $\widehat{V}$, in a way that will be clarified.
The procedure selects any winner of this tournament.

\vspace{0.5em}
Let us describe the performance of the two components. 

\subsection{The crude risk oracle}

The first component of the learning procedure is a crude risk oracle.
While its performance is far from the wanted one, it allows the learner to exclude functions that have large risks.

\begin{tcolorbox}
\begin{theorem} \label{thm:crude-risk-estimate.intro}
There are constants $c_1,c_2,c_3$ that depend only on $L$ and $\eta$ such that the following holds. Set $r$ that satisfies $\log \mathcal{M}(F,c_1 rD)\leq c_2 N$.
Then with probability at least
\[1-2\exp\left( -c_3 N \right),  \]
for every $f \in F$,
\begin{align}
\label{eq:value.oracle.intro}
\left| \Psi_{\mathcal{C}}(f) - \E (f(X)-Y)^2 \right|
\leq  \frac{1}{2} \max\left\{ r^2,  \E(f(X)-Y)^2 \right\}.
\end{align}
\end{theorem}
\end{tcolorbox}

Note that the  entropy condition  $\log \mathcal{M}(F,c_1rD) \leq c_2 N$ in the theorem is satisfied when $r\geq \frac{1}{c_1} \lambda^\ast(c_2)$, which we may assume in the context of Theorem \ref{thm:main}.

The function $\Psi_{\mathcal{C}}$ is a mean estimation procedure for functions in the class $\{ (f-Y)^2 : f\in F\}$, but as a mean estimator it is rather inaccurate: if $\E (f-Y)^2\geq r^2$ then $\Psi_{\mathcal{C}}(f)$ is an `isomorphic' guess of the mean, satisfying that 
\[
\frac{1}{2}  \E (f-Y)^2 \leq \Psi_{\mathcal{C}}(f) \leq \frac{3}{2} \E (f-Y)^2.
\]
And if $\E (f-Y)^2 <r^2$ then $\Psi_{\mathcal{C}}(f)$  only exhibits that fact without providing a better guess of what $\E (f-Y)^2$ really is.

The proof of Theorem \ref{thm:crude-risk-estimate.intro} is presented in Section \ref{sec:crude}.

\vspace{0.5em}

With  Theorem \ref{thm:crude-risk-estimate.intro} at hand, we can immediately obtain the isomorphic estimator for the noise level stated in Lemma \ref{lem:noise.estimate}.

\begin{proof}[Proof of Lemma \ref{lem:noise.estimate}]:
Recall that $\sigma^2=\inf_{f\in F} \E (f-Y)^2$ and that
\[
\widehat{\sigma}^2 = \inf_{f \in F} \Psi_{\mathcal{C}}(f).
\]
Note that in the high probability event in which \eqref{eq:value.oracle.intro} holds, we have that
$\widehat{\sigma}^2\leq 2\max\{r^2,  \sigma^2 \}$; and if $\sigma>r$ then $\widehat{\sigma}^2\geq \frac{1}{2}\sigma^2$.
Hence, if $\sigma>r$ then $\frac{1}{2}\sigma^2 \leq \widehat{\sigma}^2\leq 2\sigma^2$, and if $\sigma<r$ then $\widehat{\sigma}^2\leq 2r^2$.
\end{proof}

In particular, for $\sigma_\ast=\max\{r, \widehat{\sigma}\}$ it is evident that
\[
\min\left\{ 1, \frac{r^2}{\sigma_\ast^2}\right\}
\sim \min\left\{1, \frac{r^2}{\sigma^2}\right\}.
\]
Recall that
\[  \widehat{V}
=
\left\{f \in F: \Psi_{\mathcal{C}}(f) \leq  4  \sigma_\ast^2  \right\}
\]
is an output of the crude risk oracle.

\begin{corollary} \label{cor:crude-min.intro}
In the event in which \eqref{eq:value.oracle.intro} holds, $f^* \in  \widehat{V}$ and for every $f \in  \widehat{V}$,
\[
\E(f(X)-Y)^2 \leq 16  \max\left\{ r^2 ,  \sigma^2 \right\}.
\]
\end{corollary}

The proof of Corollary \ref{cor:crude-min.intro} follows from a straightforward application of Theorem \ref{thm:crude-risk-estimate.intro}, and is outlined in Section \ref{sec:crude}.

\subsection{The fine risk oracle}

The second main component is an accurate mean estimation procedure for functions of the form $(f(X)-Y)^2-(h(X)-Y)^2$.

\begin{tcolorbox}
\begin{theorem}
\label{thm:fine.risk.intro}
	There are constants $c_1$ and $c_2$ that depend only on $L$ and $\eta$ for which the following holds.
	For $r>\max\{ r^\ast(c_1),\lambda^\ast(c_1)\}$ there is a mean estimation procedure  $\Psi_\mathcal{L}$ with the following property.
	With probability at least
	\[1-2\exp\left( - c_2   N \min\left\{ 1, \frac{r^2}{\sigma^2} \right\} \right),\]
	for every $f,h\in F$ that satisfy
	\begin{align}
	\label{eq:fine.intro1}
	\E(h(X)-Y)^2 \leq 16  \max\left\{ r^2, \sigma^2\right\},
	\end{align}
	we have that
	\begin{align}
	\label{eq:fine.intro2}
	\left| \Psi_\mathcal{L}(f,h) - \left( \E(f-Y)^2 - \E (h-Y)^2 \right)
	\right| \leq \frac{1}{2} \max\left\{ r^2 , \|f-h\|_{L_2}^2 \right\}.
	\end{align}
\end{theorem}
\end{tcolorbox}

The proof of Theorem \ref{thm:fine.risk.intro} is presented in Section \ref{sec:fine}.

\vspace{0.5em}

With the main components set in place, the learning procedure is a tournament played between functions in $F$---following a similar path to the one used in \cite{lugosi2019risk} and \cite{mendelson2019unrestricted}.

\subsection{The tournament}

Recall that the symmetric distance oracle $d$ satisfies that for every $f,h\in F\cup \{0\}$,
\begin{align}
\label{eq:distnace.oravle.proof}
\eta^{-1}\|f-h\|_{L_2}\leq d(f,h)\leq \eta \|f-h\|_{L_2}.
\end{align}
The tournament consists of `matches' between functions in $F$.
Every match has a `home function' and a `visitor function' and the result of a match is that one of the sides wins.

\begin{tcolorbox}
\begin{itemize}
\item Only functions that belong to $\widehat{V}$ participate in the tournament.
\item  $h$ wins its home match against $f$ if
\[ \Psi_\mathcal{L}(f,h)
\geq
\begin{cases}
0 &  \text{ if }   d(f,h) \geq \eta  r, \\
-\dfrac{\eta^4}{2} r^2  &  \text{ otherwise}.
\end{cases}
\]
\item  Let $V^\ast$ be the set of functions in $\widehat{V}$ that won all their home matches.
\item The procedure selects any function in $V^\ast$.
\end{itemize}
\end{tcolorbox}

To prove Theorem \ref{thm:main}, it suffices to show that with the wanted probability, $V^\ast\neq\emptyset$ and that any $h\in V^\ast$ satisfies
\[ \|h- f^\ast\|_{L_2}\leq c(\eta) r \quad\text{ and }\quad \E (h-Y)^2 \leq \E (f^\ast-Y)^2 + (c(\eta) r)^2.\]

\begin{proof}[Proof of Theorem \ref{thm:main}]
	Fix a realisation of $(X_i,Y_i)_{i=1}^{2N}$ for which the assertions of both Theorem \ref{thm:crude-risk-estimate.intro} and Theorem \ref{thm:fine.risk.intro} hold---using the sample as described previously.
	By Corollary \ref{cor:crude-min.intro}, every function $h\in \widehat{V}$ satisfies \eqref{eq:fine.intro1}, and it follows from \eqref{eq:fine.intro2} that for every $f,h \in \widehat{V}$,
\begin{align}
\label{eq:proof.main.main}
\left| \Psi_\mathcal{L}(f,h) - \left( \E(f-Y)^2 - \E (h-Y)^2 \right)
	\right| \leq \frac{1}{2} \max\left\{ r^2 , \|f-h\|_{L_2}^2 \right\}.
\end{align}
	For every $f,h\in F$, set
	\[ \mathcal{L}(f,h)=   \E(f-Y)^2 - \E (h-Y)^2, \]
	and note that $\mathcal{L}(f,f^\ast)=\E \mathcal{L}_f$.
	Thus, by the convexity of $F$,
	\begin{align}
	\label{eq:FOC}
	\mathcal{L}(f,f^\ast) \geq \|f-f^\ast\|_{L_2}^2 \quad\text{for every } f\in F.
	\end{align}

	To prove that $f^\ast \in V^\ast$, recall that by Corollary \ref{cor:crude-min.intro}, $f^\ast\in \widehat{V}$; therefore $f^\ast$ participates in the tournament.
	Set $f\in  \widehat{V}$ and consider  the home match of $f^\ast$ against $f$.
	If  $d(f,f^\ast)\geq \eta r$, then \eqref{eq:distnace.oravle.proof} implies that  $\|f-f^\ast\|_{L_2}\geq r$, and by \eqref{eq:proof.main.main} and the convexity condition \eqref{eq:FOC},
	\[ \Psi_\mathcal{L}(f,f^\ast)
	\geq \mathcal{L}(f,f^\ast) -  \frac{1}{2} \|f-f^\ast\|_{L_2}^2
	\geq 0;\]
	hence,  $f^\ast$ wins its home match against $f$.
	
	Otherwise, if $d(f,f^\ast)<\eta r$, then  $\|f-f^\ast\|_{L_2}< \eta^2 r$ by \eqref{eq:distnace.oravle.proof}.
	It follows from  \eqref{eq:proof.main.main} that
	\[ \Psi_\mathcal{L}(f,f^\ast)
	\geq \mathcal{L}(f,f^\ast) -  \frac{1}{2} \|f-f^\ast\|_{L_2}^2
	\geq - \frac{\eta^4}{2} r^2,\]
	and again $f^\ast$ wins its home match against $f$.
	
	\vspace{0.5em}
	Now set $h\in V^\ast$ to be any other winner of the tournament.
	In particular $h$ won its home match against $f^\ast$.
	The first observation is that necessarily $d(f^\ast,h)< \eta r$.
	Indeed, if not and $d(f^\ast,h)\geq \eta r$, then $\Psi_{\mathcal{L}}(f^\ast,h)\geq 0$. At the same time, by \eqref{eq:distnace.oravle.proof}, $\|f^\ast-h\|_{L_2}\geq  r$, and thanks to \eqref{eq:proof.main.main} and \eqref{eq:FOC},
	\begin{align*}
	\Psi_\mathcal{L}(f^\ast,h)
	&\leq \mathcal{L}(f^\ast,h)  +  \frac{1}{2} \|f^\ast-h\|_{L_2}^2 \\
	&= - \mathcal{L}(h,f^\ast)  +  \frac{1}{2} \|f^\ast-h\|_{L_2}^2 \\
	&\leq -\|f^\ast -h\|_{L_2}^2 +  \frac{1}{2}\|f^\ast -h\|_{L_2}^2 <0,
	\end{align*}
	which is impossible.
	Thus $d(f^\ast,h) <\eta r$, and in particular $\|f^\ast - h\|_{L_2} < \eta^2 r$.
	Now, since $h$ won its home match against $f^\ast$, $\Psi_{\mathcal{L}}(f^\ast,h)\geq - \frac{1}{2} \eta^4 r^2$. It follows from   \eqref{eq:proof.main.main} that
	\begin{align*}
	\mathcal{L}(f^\ast,h)
	&\geq \Psi_\mathcal{L}(f^\ast,h) - \frac{1}{2} \max\left\{ r^2, \|f^\ast-h\|_{L_2}^2 \right\}  \\
	&\geq -  \eta^4 r^2,
	\end{align*}
implying that $\E (h-Y)^2\leq \E (f^\ast-Y)^2 + \eta^4 r^2$.
\end{proof}

\begin{remark}
\label{rem:procedure.minimax}
	There is an alternative method of presenting tournament procedures (see, e.g.\ \cite{audibert2011robust,baraud2017new,lecue2020robust,oliveira2025trimmed})---namely, by setting	
	\[ \widehat f \in \mathop{\rm argmin}_{f\in \widehat V} \sup_{h \in \widehat V} \Psi_\mathcal{L}(f,h).\]
	The statistical performance of this procedure matches that of the tournament and the analysis follows the same path (see \cite{lugosi2017remark} for more details). 
	Here, it follows from Theorem \ref{thm:crude-risk-estimate.intro} and Theorem \ref{thm:fine.risk.intro} that $\|\widehat f - f^\ast\|_{L_2} \lesssim r$ and $\E(\mathcal{L}_{\widehat f} | \mathcal{D}_N)\lesssim r^2$. 
\end{remark}

It is clear that at the heart of the proof of Theorem \ref{thm:main} are the two functions $\Psi_{\mathcal{C}}$ and $\Psi_{\mathcal{L}}$.
The rest of this article is devoted to their construction, which is based on a surprising combination: optimal mean estimation procedures for real valued random variables and Talagrand's generic chaining mechanism.
We expect that the impact of this combination will go well beyond the proof of Theorem \ref{thm:main}.

\section{Unorthodox chaining}
\label{sec:chaining}

\emph{Generic chaining} was introduced by M.~Talagrand for the study of certain stochastic processes.
Let $\{Z_v : v \in V\}$ be a  (centred) random process, and consider the problem of controlling $\E \sup_{v \in V} Z_v$. By creating increasingly finer approximating sets $V_s\subset V$ and setting $\pi_s v$ to be the best approximation (in some appropriate sense) of $v$ in $V_s$, it follows that
\begin{align}
\label{eq:Z.telesopic}
Z_v - Z_{\pi_0 v} = \sum_{s \geq 0} (Z_{\pi_{s+1} v}-Z_{\pi_s v}).
\end{align}
If $|V_0|=1$ then  $\E\sup_{v \in V} Z_v = \E\sup_{v \in V}( Z_v-Z_{\pi_0v})$; therefore it suffices to control the behaviour of each `link' $Z_{\pi_{s+1}v}-Z_{\pi_s v}$ in the chain corresponding to $v$ to obtain a bound on $\sup_{v \in V}( Z_v-Z_{\pi_0v})$.
Intuitively, the number of links in the $s$ stage is at most $|V_s| \cdot |V_{s+1}|$, and that cardinality has to be `balanced' with the probability that the link  $Z_{\pi_{s+1}v}-Z_{\pi_s v}$ is `large'.
As a result, if there is some compatibility between the notion of approximation used in $V$ and that probability, the chaining argument leads to an upper bound on $\E\sup_{v \in V} Z_v$ that is based on metric features of the indexing set $V$.

For a detailed exposition on chaining methods we refer the reader to Talagrand's treasured book \cite{talagrand2022upper}.

\vspace{0.5em}
Even with this somewhat vague description of generic chaining, it is clear that the key to an upper bound on $\E \sup_{v\in V} Z_v$ is the \emph{increment condition}:      an estimate on the probability that $|Z_v-Z_w|$ is much larger than the `distance' between $v$ and $w$.

\begin{definition}
	An admissible sequence of a metric space $(V,\rho)$ is a collection $(V_s)_{s\geq 0}$ of subsets of $V$ that satisfy $|V_s|\leq 2^{2^s}$ and $|V_0|=1$.
	For $v\in V$, denote by $\pi_s v$ a nearest element to $v$ in $V_s$, and set
	\[
\gamma_2(V,\rho) = \inf_{(V_s)_{s\geq 0 }} \sup_{v\in V} \sum_{s\geq 0} 2^{s/2} \rho(\pi_s v,\pi_{s+1} v),
\]
	with the infimum taken over all admissible sequences.
\end{definition}
%

What is a remarkable justification of the generic chaining mechanism is that it leads to a complete metric characterisation of the expected supremum of gaussian processes:

\begin{theorem} 
\label{thm:MM}
	There are absolute constants $c_1$ and $c_2$ for which the following hold.
	Let $\{G_v : v\in V\}$ be a centred gaussian process and set $\rho(v,w)=\|G_v-G_{w}\|_{L_2}$.
	Then
	\[  c_1 \gamma_2(V,\rho)
	\leq \E \sup_{v\in V} G_v
	\leq c_2 \gamma_2(V,\rho).\]
\end{theorem}

The upper bound follows by filling the gaps in our vague description of generic chaining,  while the lower bound is Talagrand's celebrated majorizing measures theorem \cite{talagrand1987regularity}.

\subsection{A preprocessing step}
\label{sec:chaining.admissible.seq}

The  key ingredient in the proof of the upper estimate in Theorem \ref{thm:MM} is the subgaussian tail estimate.
Such a tail estimate is simply false when considering other processes---for example---an empirical process in a heavy-tailed situation.
What is true, however, is that there are real-valued mean estimation procedures that exhibit subgaussian tails even for  general random variables.

For the remainder of this article, denote by $\psi_\delta:\R^N \to \R$ an optimal mean-estimation procedure, i.e., a data-dependent procedure that satisfies the following.

\begin{tcolorbox}
\begin{theorem}
\label{thm:optimal.1dim}
There are absolute constants $c_0,c_1$ and for any $\delta\geq 2\exp(-c_0 N)$  and any random variable $Z$ that has a finite mean and variance $\sigma_Z^2$,
\begin{equation} \label{eq:optimal-mean}
\left| \psi_\delta(Z_1,...,Z_N) - \E Z \right|
 \leq c_1 \sigma_Z \sqrt{\frac{\log(2/\delta)}{N}} \ \ \ {\rm with \ probability \ at \ least \ } 1-\delta.
\end{equation}
\end{theorem}
\end{tcolorbox}

The fact that  \eqref{eq:optimal-mean} is the best possible behaviour of a mean estimation procedure is standard; and among the examples of such optimal mean estimators are the \emph{median-of-means} (see, e.g.\ \cite{alon1999space,jerrum1886random,nemirovskij1983problem}) and the \emph{trimmed mean} (see, e.g.\ \cite{lugosi2021robust}). 
For more results on optimal mean-estimation procedures see e.g., \cite{abdalla2022covariance,catoni2012challenging,DeLeLuOl16,Min15}, and the survey \cite{lugosi2019mean}.

\begin{remark}
The identity of the mean estimation procedure that will be used in what follows  is of no importance. As far as the proof of Theorem \ref{thm:main} goes, all that matters is that $\psi_\delta$ satisfies \eqref{eq:optimal-mean}.
\end{remark}

The next ingredient in the construction of the learning procedure is an almost optimal admissible sequence $(H_s)_{s\geq 0}$ whose existence is guaranteed by Theorem \ref{thm:MM}.
Such a sequence can be constructed by solving an optimisation problem--something that can, in-principle, be done if one knows the underlying metric. Although the learner does not have access to the $L_2$ metric, the functional $d$ works equally well, because it satisfies that
\[
\eta^{-1} \|f-h\|_{L_2}  \leq d(f,h) \leq \eta \|f-h\|_{L_2} \quad\text{for every }  f,h\in H\cup\{0\}.
\]
Therefore, solving the optimisation problem for $d$ would lead to an admissible sequence $(H_s)_{s\geq 0}$ for which
\begin{align}
\label{eq:construct.admissible.sequence}
 \sum_{s\geq 0} 2^{s/2} \| \Delta_s h\|_{L_2}
\leq c(\eta)  \E \sup_{h\in H} G_h,
\end{align}
where here and throughout we set $\Delta_s h = \pi_{s+1} h - \pi_s h$.
In a similar fashion, an almost optimal admissible sequence of $(H-H)\cap r D$ can be constructed for every $r\geq 0$.

\begin{tcolorbox}
Having access to the $L_2$ distance oracle $d$ leads to  suitable admissible sequences for the classes we care about, namely $F$ and its localizations $(F-F)\cap r D$.

Moreover, thanks to those admissible sequences and the majorizing measures theorem, it is possible to identify $\E \sup_{h\in (F-F)\cap rD} G_h$ (up to multiplicative constants that depend on $\eta$).
Therefore, one may derive estimates on the fixed point condition $\E \sup_{h\in (F-F)\cap rD} G_h \leq \kappa r\sqrt{N}$.
\end{tcolorbox}

\subsection{Subgaussian mean estimators for multiplier classes}

Let $\xi \in L_4$ be a random variable (that need not be independent of $X$).

Our goal here is to design a uniform mean estimation procedure for  $\E\xi h$: the procedure receives as data the values $(\xi_i,h(X_i))_{i=1}^N$ and returns its `guess' of $\E \xi h$.

The procedure should  exhibit a subgaussian accuracy/confidence tradeoff: if $H$ were $L$-subgaussian, that is, if for every $f,h \in H \cup \{0\}$, $\|f-h\|_{\psi_2} \leq L \|f-h\|_{L_2}$, and if $\|\xi\|_{\psi_2} < \infty$, a nontrivial chaining argument (see \cite{mendelson2016upper} for the proof) shows that the following holds.
Let $H$ be a centrally symmetric\footnote{That is, if $h\in H$ then $-h \in H$.} set, put
\[ R_H = \sup_{h\in H} \|h\|_{L_2}\]
and set
\[d_*(H) = \left( \frac{ \E\sup_{h\in H} G_h }{ R_H} \right)^2  , \]
to be the \emph{critical dimension} of the set $H$.
Then for $N\geq c_0 d_\ast(H)$,  with probability at least $1-2\exp(-c_1d_*(H))$,
\begin{equation} \label{eq:subgaussian-emp-multiplier}
\sup_{h \in H} \left|\frac{1}{N} \sum_{i=1}^N \xi_i h(X_i) - \E \xi h \right|
\leq c_2(L) \frac{\|\xi\|_{\psi_2}}{\sqrt{N}} \cdot \E \sup_{h \in H} G_h.
\end{equation}

We would like to design a mean estimator that exhibits a similar error  as the empirical mean in \eqref{eq:subgaussian-emp-multiplier}, but holds for a heavy-tailed class of functions $H$ and a heavy-tailed multiplier $\xi$. Obviously, that estimator cannot be the empirical mean, whose performance deteriorates dramatically when the indexing class is not subgaussian.

As it happens, the subgaussian assumption can be relaxed considerably.

\begin{assumption}
\label{ass:H}
	$H$ is a centrally symmetric class of mean zero functions that satisfies for every $f,h \in H \cup \{0\}$,
$$
\|f-h\|_{L_4} \leq L \|f-h\|_{L_2}.
$$
Also, assume as always that the learner has access to a symmetric functional $d$ that satisfies
	\[  \eta^{-1}  \|f-h\|_{L_2} \leq d(f,h) \leq \eta  \|f-h\|_{L_2}
\]
	for every $f,h \in H \cup \{0\}$.
\end{assumption}

\begin{remark}
	Clearly the $L_4-L_2$ norm equivalence condition in Assumption \ref{ass:H} is significantly weaker than the $\psi_2-L_2$ norm equivalence that holds for subgaussian classes.
	In fact, functions in $H$ need not have $L_{4+\varepsilon}$ moments, let alone arbitrary high moments, and thus can be heavy-tailed.
\end{remark}

%
\begin{theorem}
\label{thm:multiplier}
If $H$ satisfies Assumption \ref{ass:H}, there is an absolute constant $c_1$ and a constant $c_2=c_2(L,\eta)$ for which  the following holds.
Let $\xi \in L_4$ and set $2^{s_0} \leq c_1N$.
There exists a procedure $\Phi_{\mathcal{M}}$ for which, with probability at least $1-2\exp(-2^{s_0})$,
$$
\sup_{h \in H} \left| \Phi_{\mathcal{M}}(h) - \E \xi h \right|
\leq c_2 \frac{\|\xi\|_{L_4}}{\sqrt{N}} \left(\E \sup_{h \in H} G_h + 2^{s_0/2} R_H\right).
$$
\end{theorem}

In particular, if $N \geq c_1d_*(H)$, then setting $2^{s_0} \sim d_*(H)$ we have that with probability at least  $1-2\exp(-c_3d_*(H))$,
$$
\sup_{h \in H} \left| \Phi_{\mathcal{M}}(h) - \E \xi h \right|
\leq c_4(L,\eta)\frac{\|\xi\|_{L_4}}{\sqrt{N}} \E \sup_{h \in H} G_h.
$$
This exhibits the wanted subgaussian behaviour as in \eqref{eq:subgaussian-emp-multiplier}.

The proof of Theorem \ref{thm:multiplier} follows a standard chaining argument---with one significant twist: the subgaussian increment condition holds thanks to the mean estimation procedure $\psi_\delta$ that satisfies \eqref{eq:optimal-mean} for the right choices of $\delta$.

\begin{remark}
In the proof of Theorem~\ref{thm:multiplier} (and Theorem \ref{thm:main-product}), we stop the chaining argument at level $s_1$ for which $2^{s_1} \sim N $. 
There are two reasons for that: First, the argument is based on applying \(\psi_\delta\) with the choice \(\delta_s \sim \exp(-2^s)\); that choice is admissible in Theorem~\ref{thm:optimal.1dim} only if \(s \lesssim s_1\). Second, at this scale the approximation error \(\|\pi_{s_1}h - h\|_{L_2}\) is already small, so the behaviour of \(\xi \pi_{s_1}h\) and \(\xi h\) is essentially the same: it is within the allowed error.
\end{remark}

\begin{proof}[Proof of Theorem \ref{thm:multiplier}]
Let $\alpha\geq 5$ be a well-chosen absolute constant and set $s_1$ to be the largest integer satisfying $\alpha 2^{s_1} \leq c_0 N$, where $c_0$ is the constant appearing in Theorem \ref{thm:optimal.1dim}.
In particular $s_0<s_1$.
Let $(H_s)_{s \geq 0}$ be  an admissible sequence of $H$ that satisfies \eqref{eq:construct.admissible.sequence}.

The proof consists of two steps.
First, we design a mean estimator for $\xi \pi_{s_1}h$, and later show that the means of $\xi h$ and $\xi \pi_{s_1}h$ are sufficiently close.

\vspace{0.5em}
\noindent
\emph{Step 1:}
Write
$$
\xi \pi_{s_1} h =  \xi \pi_{s_0} h+ \sum_{s=s_0}^{s_1-1} \xi \Delta_s h,
$$
and note that  for every $s\leq s_1$, the choice of $\delta_s=2\exp(- \alpha 2^s)$ is legal in  the context of Theorem \ref{thm:optimal.1dim}.
Moreover, the cardinality of the set  $\{  \xi \Delta_s h : h \in H\}$ is at most $2^{2^{s+2}}$.
 Define
$$
\Phi_{\mathcal{M}}(h) =
\psi_{\delta_{s_0}} \left(\left(\xi_i \pi_{s_0} h(X_i)\right)_{i=1}^N \right)
+\sum_{s=s_0}^{s_1-1} \psi_{\delta_s}\left(\left(\xi_i \Delta_s h(X_i)\right)_{i=1}^N \right),
$$
and the key is to show that with probability at least $1-2\exp(-2^{s_0})$, for every $h \in H$,
\begin{align} \label{eq:multi-proof-1}
 \left|\Phi_{\mathcal{M}}(h) - \E \xi \pi_{s_1}h \right|
\leq & c_1(L)  \frac{\|\xi\|_{L_4}}{\sqrt{N}} \left(\sum_{s=s_0}^{s_1-1} 2^{s/2} \|\Delta_s h\|_{L_2} +  2^{s_0/2} \|\pi_{s_0}h\|_{L_2}\right) .
\end{align}

To establish \eqref{eq:multi-proof-1}, fix $s_0 \leq s < s_1$ and examine each term $\psi_{\delta_s}((\xi_i \Delta_s h(X_i))_{i=1}^N )$---which is an optimal mean estimator of $\xi \Delta_s h$. Hence, each random variable of the form $\xi \Delta_s h$ satisfies that with probability at least $1-2\exp(-\alpha 2^s)$,
\begin{equation} \label{eq:multiplier-1}
\left| \psi_{\delta_s}\left(\left(\xi_i \Delta_s h(X_i)\right)_{i=1}^N \right) - \E \xi \Delta_s h \right|
\leq c_2 \|\xi \Delta_s h\|_{L_2} \sqrt{\frac{2^s}{N}} ,
\end{equation}
and clearly,
$$
\|\xi \Delta_s h\|_{L_2} \leq \|\xi\|_{L_4} \|\Delta_s h\|_{L_4} \leq L \|\xi\|_{L_4} \|\Delta_s h\|_{L_2}.
$$
Moreover, thanks to the high probability estimate with which \eqref{eq:multiplier-1} holds and the fact that there are at most $2^{2^{s+2}} \leq \exp(4 \cdot 2^s)$ functions of the form $\xi \Delta_s h$, we have that with probability $1-2\exp(- (\alpha-4) 2^s)$, every such function satisfies \eqref{eq:multiplier-1}.

Repeating this argument for every $s_0 \leq s \leq s_1-1$, and also for the functions $\{\xi \pi_{s_0} h : h \in H\}$, it is evident by the union bound that with probability at least $1-2\exp(- 2^{s_0})$, for every $h \in H$ and every $s_0 \leq s \leq s_1-1$,
\begin{align}
\begin{split}
\label{eq:multiplier-2}
\left| \psi_{\delta_s}\left(\left(\xi_i \Delta_s h(X_i)\right)_{i=1}^N \right) - \E \xi \Delta_s h \right|
&\leq c_2 L \|\xi\|_{L_4} \|\Delta_s h\|_{L_2} \sqrt\frac{2^s}{N} \quad\text{and}\\
\left| \psi_{\delta_{s_0}}\left(\left(\xi_i \pi_{s_0} h(X_i)\right)_{i=1}^N \right) - \E \xi \pi_{s_0} h \right|
&\leq c_2L \|\xi\|_{L_4} \|\pi_{s_0} h\|_{L_2} \sqrt\frac{2^{s_0}}{N} .
\end{split}
\end{align}
Finally,
\[
\E \xi \pi_{s_1} h
= \E \xi \pi_{s_0} h + \sum_{s=s_0}^{s_1-1} \E \xi \Delta_s h,
\]
and the wanted estimate in \eqref{eq:multi-proof-1} follows from \eqref{eq:multiplier-2}.

\vspace{0.5em}

\noindent
\emph{Step 2}:
Observe  that for every $h\in H$,
\begin{align}
\label{eq:mutiplier.diff.s1}
|\E \xi h - \E \xi \pi_{s_1} h|\leq
\frac{c_3 \|\xi\|_{L_2}}{\sqrt{N}} \sum_{s \geq s_1} 2^{s/2} \|\Delta_s h\|_{L_2}.
\end{align}
Indeed,
$$
\left| \E \xi h - \E \xi \pi_{s_1} h \right|
\leq  \|\xi(h-\pi_{s_1}h)\|_{L_1}
\leq \|\xi\|_{L_2} \|h - \pi_{s_1}h \|_{L_2}.
$$
Since $h - \pi_{s_1} h = \sum_{s \geq s_1} \Delta_s h$, it follows from  the triangle inequality and the choice of $s_1$ that
$$
\|h-\pi_{s_1} h\|_{L_2} \leq \frac{1}{2^{s_1/2}} \sum_{s \geq s_1} 2^{s_1/2} \|\Delta_s h\|_{L_2} \leq \frac{c_3}{\sqrt{N}} \sum_{s \geq s_1} 2^{s/2} \|\Delta_s h\|_{L_2},
$$
and \eqref{eq:mutiplier.diff.s1} follows.

The proof is completed by combining the estimates in \eqref{eq:multiplier-2} and \eqref{eq:mutiplier.diff.s1} and recalling that  $(H_s)_{s\geq 0}$ satisfies
\[\sum_{s \geq s_0} 2^{s/2} \|\Delta_s h\|_{L_2} \leq c_4(\eta) \E \sup_{h\in H} G_h.\qedhere\]
\end{proof}


\begin{corollary} \label{cor:multiplier-multiple}
Let $H$  be a class of mean zero functions that satisfies $L_4-L_2$ norm equivalence with constant $L$.
Let $2^{s_0}\leq c_1 N$ and  set ${\mathcal U}=\{\xi_j : 1 \leq j \leq 2\exp( 2^{s_0-1})\} \subset L_4$ to be a collection of functions.
Then there is a procedure $\Phi_{\mathcal{M}}$ for which, with probability $1-2\exp(- 2^{s_0-1})$, for every $\xi \in {\mathcal U}$,
$$
\sup_{h \in H} \left| \Phi_{\mathcal{M}}(h,\xi) - \E \xi h \right|
\leq c(L)\frac{\|\xi\|_{L_4}}{\sqrt{N}} \left(\E \sup_{h \in H} G_h + 2^{s_0/2} R_H\right).
$$
\end{corollary}

\begin{remark}
	The procedure $\Phi_{\mathcal{M}}$ has access to the values $(h(X_i))_{i=1}^N$ and $(\xi_i)_{i=1}^N$ for every $\xi\in \mathcal{U}$, and it is constructed using an admissible sequence of $H$.
	However, it has no access to the identity of each $\xi$ (or even to the identity of the set $\mathcal{U}$) beyond the given sample.
\end{remark}

\subsection{A subgaussian mean estimator for a product class}

A somewhat more involved problem is dealing with products of classes.
Let $H^1,H^2 \subset L_2(\mu)$ be centrally symmetric classes of mean zero functions, and one would like to obtain a uniform estimate on $\E f h$ for $f\in H^1$ and $h\in H^2$.
The benchmark is the performance of the empirical mean when $H^1$ and $H^2$ are subgaussian with constant $L$.
In such a scenario the right estimate requires a rather involved chaining argument.

\begin{theorem} \label{thm:subgaussian-empirical-product}
Assume that $H^1$ and $H^2$ are centrally symmetric $L$-subgaussian classes of mean zero functions and set  $2^{s_0} \leq c_1 N$.
Then with probability at least $1-2\exp(-2^{s_0})$, for every $f \in H^1$ and $h \in H^2$,
$$
\left|\frac{1}{N}\sum_{i=1}^N f(X_i)h(X_i) - \E fh \right|
\leq \frac{c_2(L)}{\sqrt N} \left( R_{H^1} \E \sup_{h \in H^2} G_h + R_{H^2} \E \sup_{f \in H^1} G_f +2^{s_0/2} R_{H^1} R_{H^2}\right).
$$
\end{theorem}
Theorem \ref{thm:subgaussian-empirical-product} is an immediate outcome of Theorem~1.13 in \cite{mendelson2010empirical}.

As it happens, one may replace the empirical mean by a better mean estimator and obtain the same subgaussian bound --- even if the classes merely satisfy $L_4-L_2$ norm equivalence rather than the $\psi_2-L_2$ norm equivalence needed in Theorem \ref{thm:subgaussian-empirical-product}.

\begin{theorem} \label{thm:main-product}
If $H^1$ and $H^2$ satisfy Assumption \ref{ass:H}, then there is an absolute constant $c_1$ and a constant $c_2=c_2(L,\eta)$ for which  the following holds.
For every $2^{s_0} \leq c_1N$, there is a procedure  $\Psi_{\mathcal{Q}}$ for which, with probability at least $1-2\exp(-2^{s_0})$,  for every $f\in H^1$ and $h\in H^2$,
\begin{align*}
\left| \Psi_{\mathcal{Q}} (f,h) - \E fh\right|
\leq  \frac{c_2}{\sqrt N} \left( R_{H^1} \E \sup_{h \in H^2} G_h + R_{H^2} \E \sup_{f \in H^1} G_f +2^{s_0/2} R_{H^1} R_{H^2}\right).
\end{align*}
\end{theorem}
\begin{proof}
To simplify notation, for any function $a=a(X)$, we write $\psi_\delta(a)$ instead of \linebreak $\psi_\delta((a(X_i))_{i=1}^N)$.

Set $\alpha\geq 10$ to be a well-chosen absolute constant and let $s_1$ be the largest integer satisfying $\alpha 2^{s_1} \leq c_0 N$; hence $s_0<s_1$.
Set $\delta_s=2\exp(-\alpha 2^s)$ and let $(H^1_s)_{s \geq 0}$ and $(H^2_s)_{s\geq 0}$ be  admissible sequences of $H^1$ and $H^2$ that satisfy \eqref{eq:construct.admissible.sequence}.

\vspace{0.5em}

\noindent
\emph{Step 1:} We begin by designing a uniform mean estimator for the products $\pi_{s_1}f \pi_{s_1} h $.
To that end, set
\[ \Psi(f,h) = \psi_{\delta_{s_0}}\left( \pi_{s_0}f \pi_{s_0} h \right) +  \sum_{s=s_0}^{s_1-1} \psi_{\delta_{s}}\left( \pi_{s+1}f\pi_{s+1} h - \pi_{s}f\pi_{s} h \right). \]
Let us show that with probability at least $1-2\exp(-2^{s_0})$, for every $f\in H^1$ and $h\in H^2$,
\begin{align}
\label{eq:quadratic}
\left| \Psi(f,h) - \E \pi_{s_1}f\pi_{s_1} h \right|
&\leq \frac{c}{\sqrt N} \left(    R_{H^1}  \E \sup_{h \in H^2} G_h +  R_{H^2} \E \sup_{f \in H^1} G_f + \sqrt{2^{s_0}} R_{H^1} R_{H^2} \right),
\end{align}
where $c=c(L,\eta)$.

Let $s \leq s_1$ and note that $\delta_s$ is a legal choice of $\delta$ in the context of Theorem \ref{thm:optimal.1dim}.
Therefore, with probability at least $1-2\exp(-\alpha 2^s)$,
\begin{align}
\label{eq:quadratic.2}
\begin{split}
&\left| \psi_{\delta_{s}}\left( \pi_{s+1}f\pi_{s+1} h - \pi_{s}f\pi_{s} h \right) - \E\left( \pi_{s+1}f\pi_{s+1} h - \pi_{s}f\pi_{s} h \right)  \right|  \\
& \qquad \leq c_1 \|\pi_{s+1}f\pi_{s+1} h - \pi_{s}f\pi_{s} h \|_{L_2} \sqrt{\frac{2^s}{N}} .
\end{split}
\end{align}
There are at most $2^{2^{s+4}}$ functions of the form $\pi_{s+1}f\pi_{s+1} h - \pi_{s}f\pi_{s} h$, and it follows from the union bound that with probability at least $1-2\exp(-2\cdot 2^{s_0})$, \eqref{eq:quadratic.2}  holds uniformly for every $f\in H^1$, $h\in H^2$ and $s_0\leq s<s_1$ .
Moreover,
	\begin{align*}
	 \pi_{s+1}f\pi_{s+1} h - \pi_{s}f\pi_{s} h
	= (\Delta_s f) \pi_{s+1} h + (\Delta_s h)\pi_{s}f,
	\end{align*}
and by the Cauchy-Schwarz inequality and the $L_4-L_2$ norm-equivalence,
	\begin{align*}
	\| \pi_{s+1}f\pi_{s+1} h - \pi_{s}f\pi_{s} h \|_{L_2}
	&\leq \|\Delta_s f\|_{L_4} \| \pi_{s+1} h\|_{L_4} + \|\Delta_s h \|_{L_4} \|\pi_{s}f\|_{L_4} \\
	&\leq L^2\left( \|\Delta_s f\|_{L_2} R_{H^2} + \|\Delta_s h \|_{L_2} R_{H^1} \right).
	\end{align*}
Therefore, in the high probability event in which \eqref{eq:quadratic.2}  holds,
\begin{align*}
&\sum_{s=s_0}^{s_1-1} \left| \psi_{\delta_{s}}\left( \pi_{s+1}f\pi_{s+1} h - \pi_{s}f\pi_{s} h \right) - \E\left( \pi_{s+1}f\pi_{s+1} h - \pi_{s}f\pi_{s} h \right)  \right|  \\
& \qquad \leq  \frac{ c_1 L^2  }{ \sqrt N} \sum_{s=s_0}^{s_1-1}  2^{s/2}  \left(  R_{H^2}\|\Delta_s f\|_{L_2} + R_{H^1} \|\Delta_s h \|_{L_2} \right)  \\
&\qquad \leq  \frac{ c_2(L,\eta)   }{ \sqrt N}  \left(  R_{H^2} \E \sup_{f\in H^1} G_f + R_{H^1} \E \sup_{h\in H^2} G_h \right) ,
\end{align*}
where the last inequality follows thanks to the choices of  the admissible sequences $(H^1_s)_{s \geq 0 }$ and  $(H^2_s)_{s \geq 0 }$.

	The same arguments can be used to show that with probability at least $1-2\exp(-2 \cdot 2^{s_0})$, for every $f\in H^1$ and $h\in H^2$,
	\begin{align*}
\left| \psi_{\delta_{s_0}}\left(  \pi_{s_0}f\pi_{s_0} h \right) - \E \pi_{s_0}f\pi_{s_0} h   \right|
&  \leq c_3(L) R_{H^1} R_{H^2} \sqrt{\frac{2^{s_0}}{N}},
\end{align*}
	and \eqref{eq:quadratic}  follows.

	\vspace{0.5em}
	
	\noindent
	\emph{Step 2:}
		To complete the proof, it suffices to show that for every $f\in H^1$ and $h\in H^2$,
	\begin{align}
	\label{eq:quadratic.s1}
	\left| \E fh - \E \pi_{s_1}f \pi_{s_1}h \right|
	&\leq \frac{c_4(\eta)}{\sqrt N} \left( R_{H^1}  \E\sup_{h\in H^2} G_h  +  R_{H^2} \E\sup_{f\in H^1} G_f \right).
	\end{align}
To that end, note that
	\begin{align*}
	\left| \E fh - \E \pi_{s_1}f \pi_{s_1}h \right|
	&\leq \|f-\pi_{s_1}f\|_{L_2} \|h\|_{L_2} + \|h-\pi_{s_1}h\|_{L_2} \| \pi_{s_1} f\|_{L_2},
\\
	\end{align*}
and $\|h\|_{L_2}\leq R_{H^2}$, $\|\pi_{s_1}f\|_{L_2}\leq R_{H^1}$.
Finally, using that $2^{s_1}\sim N$ and  \eqref{eq:construct.admissible.sequence},
	\[ \|f-\pi_{s_1}f\|_{L_2}
	\leq \frac{1}{2^{s_1/2}} \sum_{s\geq s_1} 2^{s/2} \|\Delta_s f\|_{L_2}
	\leq  \frac{c_5(\eta) }{\sqrt N} \E \sup_{ f\in H^1} G_f.
\]
	The analogue estimate on  $\|h-\pi_{s_1}h\|_{L_2} $ clearly follows from the same argument, showing that  \eqref{eq:quadratic.s1} holds.
\end{proof}

\section{A crude risk oracle}
\label{sec:crude}

From now on, we shall always assume the norm equivalence condition and the existence of a distance oracle, namely:
\begin{tcolorbox}
\begin{itemize}
\item For every $f,h \in F \cup \{0\}$, $\|f-h\|_{L_4} \leq L \|f-h\|_{L_2}$.
\item For every $f \in F$, $\|f-Y\|_{L_4} \leq L \|f-Y\|_{L_2}$.
\item  $d$ is symmetric and  $\eta^{-1}\|f-h\|_{L_2}\leq d(f,h)\leq \eta \|f-h\|_{L_2}$ for every $f,h \in F \cup \{0\}$.
\end{itemize}
\end{tcolorbox}
%

\begin{proof}[Proof of Theorem \ref{thm:crude-risk-estimate.intro}]
Let $\theta$ and $\gamma$ be (small) constants that  depend on $\eta$ and $L$ and that are chosen in what follows.
Moreover, set $2^{s_0} = \theta^2 N$, and assume that  $\log \mathcal{M}(F,\gamma rD)\leq 2^{s_0-1}$.

First we use the distance oracle $d$ to  construct a partition of $F$ to at most $\exp(2^{s_0-1})$ sets in the following way: let $(u_j)$ be a maximal $\eta \gamma r$-separated subset of $F$ with respect to $d$.
Therefore, it is an $\gamma r$-separated set with respect to $L_2$ and hence its cardinality is at most $\exp(2^{s_0-1})$.
Set $U_j$ to be the corresponding partition of $F$ to disjoint sets whose `centres' are $u_j$.
It follows that every $f\in F$ belongs to a unique set $U_j$; thus $d(u_j,f)\leq  \eta\gamma r$, implying that
\[\|f-u_j\|_{L_2} \leq  \eta^2 \gamma r.\]

Let $\psi_{\delta_{s_0}}$ be the optimal mean estimator as in Theorem \ref{thm:optimal.1dim} with $\delta_{s_0}=2\exp(-2^{s_0})$.
It follows from the union bound that with probability at least $1-2\exp(-2^{s_0-1})$, for every $1\leq j \leq \exp(2^{s_0-1})$, 
\begin{align}
\label{eq:high.prob.crude}
\left| \psi_{\delta_{s_0}}\left( (u_j-Y)^2  \right) - \E (u_j- Y)^2 \right|
\leq c_1 \sqrt\frac{2^{s_0}}{N} \| (u_j- Y)^2 \|_{L_2}
\leq \frac{1}{8}   \| u_j- Y\|_{L_2}^2, 
\end{align}
where we have also used norm equivalence in the second inequality and that $2^{s_0}= \theta^2 N$ for $\theta= 1/(8c_1L^2)$.
For $f\in F$, let $u_j$ be the centre of the unique $U_j$ to which  $f$ belongs, and set
\[ \Psi_{\mathcal{C}}(f) = \psi_{\delta_{s_0}}\left( (u_j-Y)^2 \right).\]
Thus, in the event in which \eqref{eq:high.prob.crude} holds, we have that  for every $f\in F$, 
\begin{align*}
 \left|  \Psi_{\mathcal{C}}(f)  - \E (f-Y)^2 \right|
&\leq   \frac{1}{8} \E( u_j- Y)^2 + \left| \E( u_j- Y)^2 -  \E( f- Y)^2 \right| \\
&\leq   \frac{1}{8}  \E( f- Y)^2 +  \frac{9}{8}  \left| \E( u_j- Y)^2 - \E( f- Y)^2\right|
=(1) .  
\end{align*}
Therefore, if  $\E( f- Y)^2$ and $\E( u_j- Y)^2$ are sufficiently close, namely  if
\begin{align}
\label{eq:to.show.est.sigma}
\left| \E (u_j-Y)^2 - \E (f-Y)^2 \right|  \leq \frac{1}{9} \E (f-Y)^2  + \frac{1}{4}r^2,
\end{align}
then  $(1) \leq \frac{1}{2}\max\{ r^2 , \E (f-Y)^2\}$.

To show  that \eqref{eq:to.show.est.sigma} holds, observe that
\[(u_j-Y)^2 - (f-Y)^2  = (u_j-f)^2 +2 (u_j-f)(f-Y),\]
and thus
\begin{align*}
\left| \E (u_j-Y)^2 - \E (f-Y)^2 \right|
&\leq  \| u_j-f\|_{L_2}^2 +2 \|u_j-f\|_{L_2} \|f-Y\|_{L_2} \\
&\leq  (\eta^2 \gamma r)^2  +2 \eta^2 \gamma r \|f-Y\|_{L_2} \\
&\leq  (\eta^2 \gamma r)^2  +9 ( \eta^2 \gamma r)^2  + \frac{\|f-Y\|_{L_2}^2}{9},
\end{align*}
where in the last inequality we have used the elementary fact $ab\leq \frac{t}{4} a^2 +  \frac{1}{t} b^2$ for  $a,b,t>0$.
Thus, if $\gamma=\gamma(\eta)$ is a sufficiently small constant, \eqref{eq:to.show.est.sigma} follows
\end{proof}

\begin{proof}[Proof of Corollary \ref{cor:crude-min.intro}]
	Recall that $\sigma^2=\inf_{f\in F} \E (f-Y)^2$.
	Fix a realisation of $(X_i,Y_i)_{i=1}^N$ for which the assertion of Theorem \ref{thm:crude-risk-estimate.intro} holds, i.e.,  for every $f\in F$,
	\[\left| \Psi_{\mathcal{C}}(f) - \E (f(X)-Y)^2 \right|
\leq \frac{1}{2} \max\left\{ r^2, \E(f(X)-Y)^2 \right\}.
\]
	It follows that
	if $\E (f(X)-Y)^2 > r^2$, then
\[ \frac{1}{2}\E (f(X)-Y)^2  \leq \Psi_{\mathcal{C}}(f) \leq 2 \E (f(X)-Y)^2,\]
and if $\E (f(X)-Y)^2 \leq r^2$,  then
\[ \Psi_{\mathcal{C}}(f) \leq 2 r^2.\]
Therefore, setting
\[ \widehat{\sigma}^2 = \inf_{f\in F} \Psi_{\mathcal{C}}(f),\]
	it is evident that if $\sigma\geq r$ we have that  $\frac{1}{2}\sigma^2\leq \widehat{\sigma}^2 \leq 2 \sigma^2$, and  if $\sigma<r$, then $\widehat{\sigma}^2 \leq 2 r^2$.
	Let $\sigma_\ast=\max\{r,\widehat\sigma\}$ and observe that $f^\ast\in \widehat{V}$, i.e.\ $\Psi_{\mathcal{C}}(f^\ast) \leq 4 \sigma_\ast^2$.
	Indeed, if $\sigma>r$, then
	\[\Psi_{\mathcal{C}}(f^\ast)\leq 2 \sigma^2 \leq 4\widehat{\sigma}^2,\]
	 and if  $\sigma\leq r$, then $\Psi_{\mathcal{C}}(f^\ast)\leq 2r^2$.
	
	Finally, to show that if $f\in  \widehat{V}$, then $\E (f-Y)^2 \leq 16 \max\{r^2,\sigma^2\}$, we may assume that  $\E (f-Y)^2> r^2$ and therefore
	\begin{align*}
	\E (f-Y)^2
	\leq 2\Psi_{\mathcal{C}}(f)
	&\leq 2 \cdot 4 \max\{r^2, \widehat{\sigma}^2\}
	\leq 16 \max\{ r^2 , \sigma^2\},
	\end{align*}
	as claimed.
\end{proof}

\section{A mixture mean estimation procedure}
\label{sec:fine}

Finally we have all the ingredients needed for the proof of Theorem \ref{thm:fine.risk.intro}, but before we proceed we need to make a short detour:
Set $U=W=F-F$, and $V=F$.
For every  $f,h\in F$, let $u=w=f-h$ and note that $u \in U$ and $w\in W$. Also, set $v=h$. Hence,
\[
(f-Y)^2 - (h-Y)^2 =u(w+2(v-Y)).
\]
We will design a mean estimation procedure that holds uniformly for all the  functions in the class
$$
\left\{ u(w+2(v-Y)) : u \in U, \ w \in W, \ v \in V\right\}.
$$
The procedure is given the identities of $u,w,v$ and the second half of the sample $(X_i,Y_i)_{i=N+1}^{2N}$, and returns its best guess of the mean $\E u(w+2(v-Y))$.

As will become clear immediately, that leads to  a uniform mean estimation procedure for all the functions of the form $(f-Y)^2- (h-Y)^2$.

\begin{theorem}
\label{thm:mixture}
There are absolute constants $c_1,c_2$ and a constant $c_3=c_3(L,\eta)$ such that the following holds.
Let $\theta\in(0,1)$, set $2^{s_0}\leq c_1 \theta^2 N \min\{ 1, \frac{r^2}{\sigma^2}\}$ and consider $r$ for which
\begin{align}
\label{eq:mixture.condition.r}
\E \sup_{h \in (F-F) \cap r D} G_h
&\leq \theta \sqrt{N} r\min\left\{1,\frac{r}{\sigma}\right\} \quad\text{and}\quad
\log \mathcal{M}(F,rD) \leq 2^{s_0-1}.
\end{align}
Then there exists a mean estimation procedure $\Psi_\ast$ that satisfies, with probability at least $1-2\exp(-c_2 2^{s_0})$,
for every $ u,w\in F-F$ and $v\in F$,
\begin{align*}
& \left| \Psi_\ast(u,w,v)  - \E u(w+2(v-Y))  \right|\\
& \quad \leq c_3 \theta \left( \max\{r, \|u\|_{L_2}\}  \max\{r, \|w\|_{L_2}\} +   r\max\{r, \|u\|_{L_2}\}  \frac{ r + \|v-Y\|_{L_2} }{\max\{r,\sigma \} } \right).
\end{align*}
\end{theorem}

\begin{remark}
Note that the first condition in \eqref{eq:mixture.condition.r} just means that $r \geq \max\{r_{\mathcal{Q}}(\theta),r_{\mathcal{M}}(\theta)\}$.
\end{remark}

Let us show how the assertion of Theorem \ref{thm:mixture} leads to the proof of Theorem \ref{thm:fine.risk.intro}.

\begin{proof}[Proof of Theorem \ref{thm:fine.risk.intro}] Let $(X_i,Y_i)_{i=1}^{2N}$ be the given sample. Its first half is used to compute the crude risk estimator (considered in Theorem \ref{thm:crude-risk-estimate.intro} and Lemma \ref{lem:noise.estimate}), and it outputs the value of  $\sigma_\ast$.
The second half of the sample is used to compute $\Psi_{\mathcal{L}}$, defined in what follows,  with a choice of $2^{s_0}$ that is determined by $\sigma_\ast$.

Fix a realisation of the sample $(X_i,Y_i)_{i=1}^N$ for which the assertion of Theorem \ref{thm:crude-risk-estimate.intro} holds. This event has probability $1-\exp(-c_0N)$, and we have that
$$
\min\left\{ 1,\frac{r^2}{\sigma_\ast^2} \right\} \sim \min \left\{ 1,\frac{r^2}{\sigma^2} \right\}.
$$
Next, in the setting of Theorem \ref{thm:mixture} and using its notation, let $\theta=1/(12 c_3)$ and set
\[
2^{s_0} \sim  \theta^2 N\min\left\{1,\frac{r^2}{\sigma_\ast^2} \right\}
	\leq c_1 \theta^2 N\min\left\{1,\frac{r^2}{\sigma^2} \right\};
\]
Thus  $r$ satisfies \eqref{eq:mixture.condition.r}.
	
For $f,h\in F$ let
\[
\Psi_{\mathcal{L}}(f,h)
	= \Psi_\ast(f-h,f-h,h),
\]
where $\Psi_\ast$ is computed using $(X_i,Y_i)_{i=N+1}^{2N}$.
	If we set $u=w=f-h$ and $v=h$, then
	\[ \E u(w+2(v-Y))
	= \E (f-Y)^2 - \E (h-Y)^2,\]
	and in the high probability event in which the assertion of Theorem \ref{thm:mixture} holds,
	\begin{align*}
& \left| \Psi_{\mathcal{L}}(f,h)  -  \left( \E (f-Y)^2 - \E (h-Y)^2 \right)   \right|\\
& \quad \leq \frac{1}{12}  \left( \max\{r^2, \|f-h\|_{L_2}^2\} + r  \max\{r, \|f-h\|_{L_2}\}   \frac{ r + \|h-Y\|_{L_2} }{\max\{r,\sigma\}} \right).
	\end{align*}

	Finally, if  $\|h-Y\|_{L_2}^2\leq  16 \max\{r^2, \sigma^2\}$, then using the inequality $a\max\{a,b\}\leq \max\{a^2,b^2\}$, it follows that
	\begin{align*}
& \left| \Psi_{\mathcal{L}}(f,h)  -  \left( \E (f-Y)^2 - \E (h-Y)^2 \right)   \right|\\
&\quad  \leq \frac{1}{12}   \max\{r^2, \|f-h\|_{L_2}^2\}  \left( 1+ \frac{ r + 4 \max\{r, \sigma\}}{\max\{r, \sigma\}} \right)
\\
 & \quad \leq \frac{1}{2}  \max\{r^2, \|f-h\|_{L_2}^2\} ,
	\end{align*}
	as claimed.
\end{proof}

Let us turn to the definition of the procedure $\Psi_\ast$.
The starting point is to use  the distance oracle and the entropy estimate in \eqref{eq:mixture.condition.r} to construct a partition of  $F$ to at most $\exp(2^{s_0-1})$ sets $V_j$ with centres $v_j$.
Every $v \in F$ belongs to a unique $V_j$, and
\[
\|v-v_j\|_{L_2} \leq \eta^2 r=r_0.
\]
Moreover, for every $u,w$ and $v$,
$$
u(w+2(v-Y))=uw+ 2u(v-v_j)+2u(v_j-Y).
$$
Set $\Psi_{\mathcal{Q}}$ to be the product  mean estimation procedure for the classes $H^1=H^2=(F-F) \cap r_0 D$, and let $\Psi_{\mathcal{M}}$  be the  mean estimation procedure for the class $H=(F-F) \cap r_0 D$ and the $\exp(2^{s_0-1})$ multipliers  $v_j - Y$.

Next, recall that $r \leq r_0$, for $\tilde{u} \in F-F$ set $\alpha(\tilde{u})=\max\{r, d(\tilde{u},0)\}$, and note that
$$
\frac{r \tilde{u}}{\alpha(\tilde{u})} \in (F-F) \cap r_0 D.
$$

With that in mind, for $\tilde{w}, \tilde{u} \in F-F$ set
$$
u = \frac{r \tilde{u}}{\alpha(\tilde{u})}, \ \ {\rm and} \ \   w = \frac{r \tilde{w}}{\alpha(\tilde{w})},
$$
and define the following mean estimation procedures for $\tilde{w}, \tilde{u} \in F-F$ and $v \in F$.
Let
$$
\Psi_1(\tilde{u},\tilde{w})
=\frac{\alpha(\tilde{u})}{r} \cdot \frac{\alpha(\tilde{w})}{r} \Psi_{\mathcal{Q}}(u,w).
$$
In other words, scale-down $\tilde{u}$ and $\tilde{w}$ into $(F-F) \cap r_0 D$; then use the product mean estimation procedure  in  $(F-F) \cap r_0 D$; and finally, re-scale using the inverse factor.
Since $\E uw$ is homogeneous, the result is an estimator of $\E \tilde{u} \tilde{w}$.

In a similar fashion, set
$$
\Psi_2(\tilde{u},v) = \frac{\alpha(\tilde{u})}{r} \Psi_{\mathcal{Q}}(u,v-v_j),
$$
with the scaling and re-scaling only taking place for $\tilde{u}$.
Finally, let
$$
\Psi_3(\tilde{u},v) = \frac{\alpha(\tilde{u})}{r} \Psi_{\mathcal{M}}(u,v_j-Y)
$$
be a uniform estimator for functions of the form $\tilde{u}\xi_j$ for the $\exp(2^{s_0-1})$ multipliers $\xi_j=v_j-Y$.

Now for every $\tilde{u},\tilde{w} \in F-F$ and $v \in F$, set
$$
\Psi_*(\tilde{u},\tilde{w},v) = \Psi_1(\tilde{u},\tilde{w})+2\Psi_2(\tilde{u},v)+2\Psi_3(\tilde{u},v).
$$

\begin{remark}
	The somewhat involved definition of $\Psi_\ast$ is due to the fact that one cannot directly apply the product estimation procedure to $2u(v-Y)$. The reason is that the functions $v-Y$ need not have zero mean, which is a crucial assumption in Theorem~\ref{thm:main-product}.
	That is precisely where the anchor points $v_j$ come into play, as the multiplier estimation procedure 	(Corollary~\ref{cor:multiplier-multiple}) does not require the multipliers to have zero mean.
	Note that $v-v_j$ need not be rescaled: thanks to the way they are constructed, their $L_2$-norm is bounded by $r_0$.
\end{remark} 

The key to the performance of $\Psi_\ast$ is its behaviour on the scaled-down versions, which we explore next.

\begin{proposition}
\label{prop:mixture.localised}
In the setting of Theorem \ref{thm:mixture} and using its notation, we have that with probability at least $1-2\exp(-c_2 2^{s_0})$,
for every $u,w \in (F-F) \cap r_0D $ and $v \in F$,
\begin{align*}
\left|\Psi_{\mathcal{Q}}(u,w) - \E uw \right|
&\leq c_3\theta r^2, \\
\left| \Psi_{\mathcal{Q}}(u,v-v_j) - \E  u(v-v_j) \right|
&\leq c_3 \theta r^2,\quad \text{and}\\
\left| \Psi_{\mathcal{M}}(u,v_j-Y) - \E u(v_j-Y) \right|
&\leq c_3 \theta r^2 \cdot  \frac{r+\|v-Y\|_{L_2}}{\max\{r,\sigma\}}.
\end{align*}
\end{proposition}

\begin{proof}
Observe that the choice $2^{s_0} \leq c_1 N$ is  valid in the context of  Theorem \ref{thm:main-product}.
Therefore, with probability at least $1-2\exp(- 2^{s_0})$, for every $u,w \in (F-F) \cap r_0 D$,
\[
\left|\Psi_{\mathcal{Q}}(u,w) - \E uw \right|
\leq  \frac{ c_2}{ \sqrt N} \left( r_0 \E \sup_{h \in (F-F) \cap r_0 D}G_h +  2^{s_0/2} r_0^2\right)
=(1),
\]
for $c_2=c_2(L,\eta)$.
Moreover, for every $v\in F$, $v-v_j\in (F-F)\cap r_0 D$---with $v_j$ being the centre in the unique set in the partition $(V_j)$ to which $v$ belongs.
Hence, on the same event, for every $u\in (F-F)\cap r_0D$ and $v \in F$,
$$
\left| \Psi_{\mathcal{Q}}(u,v-v_j) - \E  u(v-v_j) \right|
\leq   (1).
$$

Next, by Corollary \ref{cor:multiplier-multiple} and  since
\[\|v_j - Y\|_{L_4} \leq c_3(L) \left(r_0+\|v-Y\|_{L_2}\right), \]
we have that with probability at least $1-2\exp(-2^{s_0-1})$, for every $u\in (F-F)\cap r_0D$ and $v_j$, $1\leq j\leq \exp(2^{s_0-1})$,
\begin{align*}
\left| \Psi_{\mathcal{M}}(u,v_j-Y) - \E u(v_j-Y) \right|
& \leq c_4 \frac{\|v_j-Y\|_{L_4}}{\sqrt{N}} \left(\E \sup_{h \in (F-F) \cap r_0 D} G_h + 2^{s_0/2} r_0 \right)
\\
& \leq c_5 \frac{r_0+\|v-Y\|_{L_2}}{\sqrt{N}} \left(\E \sup_{h \in (F-F) \cap r_0 D} G_h + 2^{s_0/2} r_0 \right)
=(2),
\end{align*}
for constants $c_4$ and $c_5$ that depend on $L,\eta$.

\vspace{0.5em}
To control the error terms $(1)$ and $(2)$, by the condition on $r$ in \eqref{eq:mixture.condition.r} and since $r\leq r_0\leq \eta^2 r$, it is evident that
\begin{align}
\label{eq:proof.some.equation.local.gamma}
\E \sup_{h \in (F-F) \cap r_0 D} G_h
\leq c_6(\eta) \theta \sqrt{N} r \min\left\{1,\frac{r}{\sigma}\right\},
\end{align}
and by the choice of $s_0$,
\begin{align}
\label{eq:proof.some.equaiton.for.s0}
2^{s_0}\lesssim \theta^2 N\min\left\{ 1,\frac{r^2}{\sigma^2} \right\}.
\end{align}
Hence,
\begin{align*}
(1)
\leq c_7(L,\eta) \theta r^2.
\end{align*}
Finally, using  \eqref{eq:proof.some.equation.local.gamma}, \eqref{eq:proof.some.equaiton.for.s0}, the fact that $r_0\leq \eta^2 r$, and the elementary observation that $\min\{1,\frac{r}{\sigma}\} = \frac{r}{\max\{r,\sigma\}}$, we have
\begin{align*}
(2)
& \leq c_{8} (r_0+\|v-Y\|_{L_2})\cdot  \theta  r \min\left\{1, \frac{r}{\sigma} \right\}    \\
&\leq   c_{9}\theta   r^2 \cdot \frac{ r+\|v-Y\|_{L_2} }{\max\{r,\sigma\}}
\end{align*}
for constants $c_{8}$ and $c_{9}$ that depend only on $L,\eta$.
\end{proof}

\begin{proof}[{\bf Proof of Theorem \ref{thm:mixture}}]
Fix a realisation of the high probability event in which the assertion of Proposition \ref{prop:mixture.localised} holds.

To control the error of $\Psi_\ast$ (when $\alpha(\tilde{u}) >r$ and/or  $\alpha(\tilde{w}) >r$), recall the scaling-down
$$
u = \frac{r \tilde{u}}{\alpha(\tilde{u})} \ \ {\rm and} \ \   w = \frac{r \tilde{w}}{\alpha(\tilde{w})}
$$
implying that  $u,w \in (F-F) \cap r_0D$. Hence
\begin{align*}
\left|\frac{\alpha(\tilde{u}) \alpha(\tilde{w})}{r^2} \cdot  \left( \Psi_{\mathcal{Q}}(u,w) - \E uw \right) \right|
\leq & \frac{\alpha(\tilde{u}) \alpha(\tilde{w}) }{r^2} \cdot c(L)\theta r^2
\\
= & c(L)\theta \alpha(\tilde{u}) \alpha(\tilde{w}).
\end{align*}
But
$$
\frac{\alpha(\tilde{u}) \alpha(\tilde{w}) }{r^2} \E uw = \E \tilde{u} \tilde{w}, \ \ {\rm and} \ \ \frac{\alpha(\tilde{u}) \alpha(\tilde{w})}{r^2} \Psi_{\mathcal{Q}}(u,w) = \Psi_{1}(\tilde{u},\tilde{w}).
$$
Thus, for every $\tilde{u}, \tilde{w} \in F - F$,
\begin{align*}
|\Psi_1(\tilde{u},\tilde{w}) - \E \tilde{u}\tilde{w} |
 \leq & c_1(L,\eta)\theta \alpha(\tilde{u}) \alpha(\tilde{w})
\\
\leq & c_2(L,\eta) \theta \max\{r, \|\tilde{u}\|_{L_2}\} \cdot \max\{r, \|\tilde{w}\|_{L_2}\}.
\end{align*}

A similar analysis for the other two terms, (keeping in mind that $v-v_j$ is not scaled-down), shows that
\begin{align*}
|\Psi_2(\tilde{u},v) - \E  \tilde{u}(v-v_j) |
&\leq c_3(L,\eta)  \theta r \cdot \max\{r, \|\tilde{u}\|_{L_2}\} \quad\text{and} \\
|\Psi_3(\tilde{u},v) - \E \tilde{u}(v_j-Y)|
&\leq c_3(L,\eta)  \theta r \cdot \max\{r, \|\tilde{u}\|_{L_2}\} \frac{r_0+\|v-Y\|_{L_2}}{\max\{r,\sigma\}}.
\end{align*}

The proof follows because
\[ \E \tilde{u}(\tilde{w}+2(v-Y))
= \E \tilde{u}\tilde{w}+ 2\E \tilde{u}(v-v_j)+2\E \tilde{u}(v_j-Y). \qedhere \]
\end{proof}

\section{Additional remarks}

\subsection{Comparison to previous work}
\label{sec:comparison}

Let us compare our results with closely related contributions, namely \cite{lugosi2019risk} (see Theorem~1 and Corollary~2), \cite{oliveira2025trimmed} (which relaxes Assumption~\ref{ass:LugMen} and allows for adversarial contamination), and \cite{lecue2020robust} (which considers regularized learning problems). 

All these works lead to bounds on the excess risk (and thus on the sample complexity) in terms of Rademacher complexities.
In that respect, the guarantees are weaker than the ones we establish here.
There are also some differences in the setup. 
While \cite{lecue2020robust,lugosi2019risk,oliveira2025trimmed} do not require the covering condition leading to $\lambda^\ast$, that condition can be verified easily under mild assumptions (see Lemma~\ref{lem:r.almost.implies.lambda}).
Moreover, the procedures in \cite{lecue2020robust,oliveira2025trimmed} are based on a minimax formulation instead of a tournament (see Remark~\ref{rem:procedure.minimax}) and receive as input only the confidence parameter $\delta$.
In contrast, the  procedure introduced here (as well as the one from \cite{lugosi2019risk}) also requires a target accuracy, which depends on the (unknown) noise level $\sigma$. 
The need for the target accuracy  is not a significant limitation, as $\sigma$ can be estimated sufficiently well in our setting (see Lemma~\ref{lem:noise.estimate}) and the required estimate is coarse (up to multiplicative constants).

The main difference between our work and \cite{lecue2020robust,lugosi2019risk,oliveira2025trimmed} is Assumption~\ref{ass:distance.oracle}: the availability of a crude distance oracle.
This assumption is central to the analysis presented here.
Whether it is truly necessary remains  open. 
If it is, then the meaning of Theorem~\ref{thm:main} is that a sample complexity estimate that is based on the behaviour of the limiting  gaussian process boils down to finding  such an oracle---a fact that we find completely unexpected.

Regarding the technical novelty, we note that a related non-standard generic chaining approach was used in \cite{bartl2025uniform} for uniform mean estimation. 
Here, we also use a uniform mean estimation argument, but for a significantly more complicated object---namely the excess loss class---which is nonlinear and not homogeneous. 
Moreover, the error must be controlled via the structure of localised versions of $F$, rather than of the class $\{(Y-f(X))^2 - (Y-f^\ast(X))^2 : f\in F\}$.
As a result, the analysis is considerably harder than in  \cite{bartl2025uniform}.

\subsection{Rademacher complexities vs limiting gaussian fixed points}
\label{sec:lin.reg.B1d}

Consider linear regression in the set $T = B_1^d = \{ t \in \R^d : \sum_{i=1}^d |t_i| \leq 1 \}$, and with $X$ being an isotropic random vector\footnote{That is, its mean is zero and its covariance matrix is the identity.}. 
Even in this simple setting, the Rademacher complexities and the fixed points of the limiting gaussian process may differ substantially. 
To see that, consider the sample complexity bounds (see Corollary~\ref{cor:main-LugMen} and Corollary~\ref{cor:main}) and since $T=B_1^d$, the interesting range is $N\ll d$.
In that range, and because one always has $\varepsilon \gtrsim \frac{1}{N}$, it suffices to restrict our attention to $\varepsilon \in (\frac cd, 1)$, which simplifies the notation without affecting the main conclusions.

We begin with the estimate on $N_0(\varepsilon)$ from Corollary~\ref{cor:main}. 
We show in Lemma~\ref{lem:B1d.gauss} that $	\lambda^\ast(c_1)  \leq c_2\sqrt\varepsilon$ if and only if
	$ N \geq c_3 \max\{1 , \frac{\sigma^2}{\varepsilon}\} \log(ed \varepsilon)$
	and  $r^\ast(c_1)  \leq c_2\sqrt\varepsilon$ if and only if 
	$ N \geq c_4 \max\{\frac{1}{\varepsilon} , \frac{\sigma^2}{\varepsilon^2}\} \log(ed \varepsilon)$.
Therefore, 
\[
N_0(\varepsilon) \lesssim
\max\left\{ \frac{1}{\varepsilon}, \frac{\sigma^2}{\varepsilon^2}\right\}
\log(ed \varepsilon).
\]
In particular, when $\varepsilon$ is a (small) constant, this reduces to
\[
N_0(\varepsilon) \lesssim \max\{1, \sigma^2\} \log(ed).
\]

In contrast, bounds on $N_0(\varepsilon)$ that are based on Rademacher complexities (following Corollary~\ref{cor:main-LugMen} and, e.g.,~\cite{lecue2020robust,oliveira2025trimmed})  deteriorate significantly. 
Indeed, again considering  the case of constant $\varepsilon$, one can construct an isotropic random vector $X$ for which  Assumption~\ref{ass:LugMen} is satisfied and $r^{\rm rad}_{\mathcal{Q}}(c_1)> c_2 \sqrt{\varepsilon}$ only when $N \lesssim d^{1/4}$; hence, the best estimate one can hope for is 
\[ N_0(\varepsilon) \gtrsim d^{1/4},
\]
(see Lemma~\ref{lem:B1d.heavy} for the details).

\subsection{Lower bounds for the sample complexity}
\label{sec:lower}

It is instructive to compare the upper bounds established here  with known  lower bounds.
The latter are in the minimax sense over a class of admissible targets, and as we explain immediately, that class of admissible targets is much smaller than the class of targets for which the upper bounds hold. 

First, note that the  estimate in Theorem~\ref{thm:main},
\begin{align}
\label{eq:lower.estim}
\P \left( \|\widehat f - f^\ast\|_{L_2} \leq r \right)
 \geq 1- 2\exp\left( -cN \min\left\{1, \frac{r^2}{\sigma^2}\right\}\right),
\end{align}
is essentially one-dimensional.
It is well known that \eqref{eq:lower.estim} is optimal  under minimal assumptions; see, e.g., \cite{lecue2013learning}.
Accordingly, we focus on the dependence of the error level $r$ on the complexity of $F$.

The current state-of-the-art lower bounds are based on local covering numbers (as opposed to the upper bounds established here, which are expressed in terms of the expected supremum of the limiting gaussian process over the same localised set).
To formulate these bounds, for every $h \in F$, set
\[F_{h,r} = \{ f\in F : \|f-h\|_{L_2}\leq r\}\]
 and for $\kappa,\gamma>0$, put
\[ \lambda_{\rm low}^\ast (\kappa,\gamma)= \sup_{h\in F} \inf\left\{ r>0 :   \log \mathcal{M}(F_{h,r}, \gamma rD)  \leq \kappa N\min\left\{1, \frac{r^2}{\sigma^2} \right\} \right\}.
\]
Note that  $\lambda_{\rm low}^\ast (\kappa,\gamma)$ is essentially equal to $\max\{ \lambda_{\mathcal{Q}}(\sqrt\kappa,\gamma),  \lambda_{\mathcal{M}}(\sqrt \kappa/\sigma,\gamma) \}$ (see  \eqref{eq:covering-fixed-point-1} and  \eqref{eq:covering-fixed-point-2}); the difference being that $\lambda^\ast_{\rm low}$ is defined using $F_{h,r}$ while $\lambda_{\mathcal{Q}}$ and $\lambda_{\mathcal{M}}$ use $(F-F)\cap rD$.

Let $w$ be a standard gaussian random variable that is independent of $X$.
For every $\sigma>0$, denote the set of `trivial targets' at level $\sigma$ by 
\[ \mathcal{Y}_\sigma = \{ f_0(X) : f_0 \in F \} \cup \{ f_0(X) + \sigma w : f_0 \in F \} ;\]
in other words, targets that are either in $F$, or a perturbation of such functions with independent additive gaussian noise whose variance is $\sigma^2$.  
 
The following is an immediate consequence of \cite{lecue2013learning,mendelson2017local}:

\begin{theorem}
There exist absolute constants $\kappa,\gamma$ for which the following holds.
Let $F \subset L_2(\mu)$ be convex and centrally symmetric. 
Then for every $\sigma>0$ and every learning procedure $\Phi$, there exists $Y\in \mathcal{Y}_\sigma$ for which with probability at least $\frac{1}{4}$, 
\[ \| \Phi  - f_0 \|_{L_2} \geq  \lambda_{\rm low}^\ast (\kappa,\gamma).\]
\end{theorem}

Note that by Sudakov's inequality (see, e.g., \cite{ledoux1991probability}), for every set $H$ and every $\varepsilon>0$, 
\[ \log \mathcal{M}(H, \varepsilon D)
\leq  c\left( \frac{\E \sup_{h\in H } G_h}{\varepsilon } \right)^2.\]
Hence, it follows that $ \lambda_{\rm low}^\ast(\kappa,\gamma) \leq r^\ast(c'(\kappa,\gamma))$.
At the same time, there are many natural examples of classes $F$ for which $\lambda^\ast\sim r^\ast$, or the gap between the two is  modest (e.g.\ logarithmic in the dimension). 
In such cases, the accuracy/confidence tradeoff in Theorem \ref{thm:main} (resp.\ the upper bound on the sample complexity in Corollary \ref{cor:main}) is sharp.

However, it is not always the case that $\lambda^\ast\sim r^\ast$, and  one can construct examples of classes $F$ for which the two differ significantly. 
At the same time, the set of admissible targets used in the context of those lower bounds is rather small---even $\bigcup_{\sigma >0} \mathcal{Y}_\sigma$ is a small set (for instance, it does not contain targets of the form $h(X)+\sigma w$ for $h \not \in F$, not to mention more complex possibilities---hence the name `trivial targets').  
This structure of the trivial targets plays a crucial role in the proof of the lower bound, and that might be the reason behind  the gap between the lower bound and the upper one---which is valid for a much richer class of potential targets. 

It remains an open question whether one can improve these lower bounds, or whether the upper bound is slightly loose (or perhaps the true sharp bound lives between the two).
One indication that lower bounds based on (local) covering numbers may be suboptimal is that they hint towards several natural candidates of  learning procedures that might be optimal---namely, restricting learning to a finite cover, or to the convex hull  or other mixtures of that cover.
Yet one can show that these approaches are suboptimal when considering the general collection of targets we are interested in here, even in light-tailed settings.
Moreover, in order to obtain an excess risk of the right order $r^2$  the choice of cover cannot be arbitrary; it has to reflect the infinitesimal geometry of (the boundary of) $F$.
It is not known if (or when) such special covers can be constructed, and to make things worse, even if such constructions do exist, they must be obtained  in a data-dependent manner.
Finding such data-dependent constructions seems to be a true challenge. 

It is important to emphasize once again that our goal here was to close the gap between the Rademacher complexities and the fixed points associated with the limiting gaussian---by showing that the estimate based on the former is loose.  Closing the remaining gap---between the new state-of-the-art upper bound and existing lower bounds is a different question altogether that is not fully understood even if the problem is completely gaussian.

\vspace{1em}

\noindent
{\bf Acknowledgements:} This research was funded by the Austrian Science Fund (FWF) [doi: 10.55776/P34743 and 10.55776/ESP31], the Austrian National Bank [Jubil\"aumsfond, project 18983] and a Presidential Young Professorship grant [‘Robust statistical learning for complex data’].

\bibliographystyle{abbrv}

\appendix
\section{The entropy estimate}
\label{sec:app.entropy}

Recall that $d_F={\rm diam}(F,L_2)$ and that for $\kappa\in(0,1)$,
\begin{align*}
r^*(\kappa)
& = \inf\left\{ r>0 : \E \sup_{h \in (F-F) \cap rD} G_h \leq \kappa \sqrt{N} r \min\left\{1, \frac{r}{\sigma}\right\} \right\}, \\
\lambda^\ast(\kappa)
&= \inf\left\{ r>0 : \log \mathcal{M}(F, rD)
\leq \kappa  N \min\left\{1, \frac{r^2}{\sigma^2} \right\} \right\},\quad\text{and} \\
\tilde{r}^{\,\ast}(\kappa) &= \inf\left\{ r>0 : \E \sup_{h \in (F-F) \cap rD} G_h
	\leq   \kappa \sqrt{N} \theta(r)  r\min\left\{1, \frac{r}{\sigma}\right\}   \right\},
\end{align*}
where $\theta(r) = 1/\max\{1,\sqrt{\log(2d_F/r)}\}$.

Let us prove the claim made previously---that these quantities are not too far from each other.
To that end, recall that trivially for any $\kappa$ we have that $\lambda^\ast(\kappa) \leq 2d_F$; indeed  $\mathcal{M}(F, 2d_F D)=1$.
Hence, the focus is only on cases where $r^\ast\leq 2d_F$.
We have the following generalization of Lemma \ref{lem:r.almost.implies.lambda}:

\begin{lemma}
\label{lem:covering.almost.automatic}
There is an absolute constant $c$ such that for every $\kappa\in(0,1)$, $\lambda^\ast(\kappa)\leq \tilde{r}^{\,\ast}(c\kappa)$.
Moreover, if $r^\ast(c\kappa)\leq 2d_F$ satisfies that 
\[ r^\ast(c\kappa) \sqrt{  \log\left( \frac{4d_F}{r^\ast(c\kappa)} \right) } \leq \sigma\]
then
\[\lambda^\ast(\kappa) \leq r^\ast(c\kappa)  \sqrt{\log\left(\frac{4d_F}{r^\ast(c\kappa)}\right)}.\]
\end{lemma}
\begin{proof}
For every $s,t>0$, set
\[H(s,t)=\sup_{f\in F} {\mathcal N}(F \cap (f+sD), tD),\]
i.e.\ $H(s,t)$ is the $t$ covering number with respect to the $L_2$ norm of the `largest' localization of $F$ of radius $s$.
Using that $F$ is convex, one may show that
$$
\log H(s,t) \leq c_0 \log\left(\frac{2s}{t}\right) \cdot \log H(4t,t),
$$
see \cite[Lemma 3.1]{mendelson2017local}.
Therefore, setting $s= d_F$,
$$
\log {\mathcal N}(F,tD) \leq c_0 \log\left(\frac{2d_F}{t}\right) \cdot \log H(4t,t).
$$
Next, note that for every $f \in F$, if $h \in F \cap (f+4tD)$, then $h=f+(h-f)$, and $h-f \in (F-F) \cap 4tD$. Therefore,
\[
F \cap (f+4tD) \subset f + \left[(F-F) \cap 4tD\right],
\]
implying that
$$
{\mathcal N}(F \cap (f+4tD),tD) \leq {\mathcal N}( (F-F) \cap 4tD,tD).
$$
By Sudakov's inequality 
$$
\log {\mathcal N}( (F-F) \cap 4tD,tD) \leq c_1 \left( \frac{ \E \sup_{h \in (F-F) \cap 4tD} G_h }{t} \right)^2,
$$
and thus,
\begin{align}
\label{eq:local.to.global.sudakov}
\log {\mathcal N}(F,tD) \leq c_2 \log\left(\frac{2d_F}{t}\right) \cdot \left( \frac{ \E \sup_{h \in (F-F) \cap 4tD} G_h }{t} \right)^2.
\end{align}

The claim that $\lambda^\ast(\kappa)\leq \tilde{r}^{\,\ast}(c\kappa)$ follows from \eqref{eq:local.to.global.sudakov}, the definitions of $\lambda^\ast$ and $\tilde{r}^{\,\ast}$, and the standard relations between packing numbers and covering numbers.

As for the other claim,  let $r^\ast(c\kappa)\leq 2 d_F$ satisfy that
\[
r=4 r^\ast(c\kappa) \sqrt{\log\left(\frac{4d_F}{r^\ast(c\kappa)}\right)} \leq  4\sigma .
\]
Then  $\frac{1}{2}r\geq r^\ast(c\kappa)$ and hence  ${\mathcal M}(F,rD) \leq {\mathcal N}(F,r^\ast(c\kappa)D)$. 
Thus, by \eqref{eq:local.to.global.sudakov} and the definition of $r^\ast(c\kappa)$,
\begin{align*}
\log {\mathcal N}(F,r^\ast(c\kappa)D)
&\leq c_2 \log\left(\frac{2d_F}{r^\ast(c\kappa)}\right) \cdot  (c\kappa)^2 N \frac{(r^\ast(c\kappa))^2}{\sigma^2}  \\
& \leq c_3  (c \kappa)^2 N  \frac{r^2}{\sigma^2}
\leq  16c_3  (c \kappa)^2 N  \min\left\{1, \frac{r^2}{\sigma^2}\right\},
\end{align*}
implying that $\lambda^\ast(\kappa) \leq r$, as claimed.
\end{proof}

\section{Proofs for Section \ref{sec:lin.reg.B1d}}
\label{sec:rad.vs.gauss.proof}

\begin{lemma}
\label{lem:B1d.gauss}
	Let $T=B_1^d$ and set $\varepsilon\in[ \frac{1}{d},1]$.
	For every $\kappa,c_0>0$  there are constants $c_1,c_2$ that depend only on $\kappa,c_0$ such that
	$\lambda^\ast(\kappa)  \leq c_0\sqrt\varepsilon$ if and only if
	$ N \geq c_1 \max\{1 , \frac{\sigma^2}{\varepsilon}\} \log(ed \varepsilon)$; and  $r^\ast(\kappa)  \leq c_0 \sqrt\varepsilon$ if and only if 
	$ N \geq c_2 \max\{\frac{1}{\varepsilon} , \frac{\sigma^2}{\varepsilon^2}\} \log(ed \varepsilon)$.
\end{lemma}
\begin{proof}
	The claim regarding  $\lambda^\ast$ follows immediately from the well known estimate on  packing numbers of $B_1^d$: for every $r\in[0,1]$,
	\[ \log \mathcal{M}(B_1^d, rB_2^d) 
	\sim
	\begin{cases}
	 \log(edr^2) & \text{if } r\geq 1/\sqrt d,\\
	 r^2 d & \text{if } r< 1/\sqrt d.
	 \end{cases} \]
	 
	As for the claim for $r^\ast$, let $\Sigma_k\subset B_2^d$  be the set of $k$-sparse vectors and put $k\sim 1/r^2$.
	Then
	 \[ (B_1^d-B_1^d)\cap rB_2^d 
	 \sim  r \Sigma_{k},\]
	 where $A\sim B$ means that $cA\subset B \subset c'A$.
	 Therefore, 
	 \begin{align*} 
	 \E \sup_{ t\in (B_1^d-B_1^d)\cap rB_2^d } \inr{G,t}
	 \sim r  \E \sup_{ t\in \Sigma_k } \inr{G,t}
	 &\sim  \begin{cases}
	 r\sqrt{k\log(ed/k)} & \text{if } k \leq d,\\
	 r\sqrt{d} & \text{if } k>d.
	 \end{cases}   
	\end{align*}
	Since $k\sim 1/r^2\leq d$, the claim follows. 
\end{proof}

Next, let us construct a random vector $X$ for which the Rademacher complexities differ substantially from their counterparts that are defined using the limiting  gaussian process.
Fix $k=k(d)\geq 2$ to be specified in what follows and let $\tilde x$ be a symmetric random variable which satisfies
 \[
 \PP(|\tilde x|=k^{1/4})=\frac{1}{k} \ \ {\rm and } \ \ \PP(|\tilde x|=1)=1-\frac{1}{k};
\]
thus    $1 \leq \| \tilde x\|_{L_2} \leq 2$ and  $\|\tilde x\|_{L_4} \leq 2$.
Let $\eps$ be a symmetric $\{-1,1\}$-valued random variable that is independent of $\tilde x$, and set  $x=\eps \frac{|\tilde{x}|}{c_0}$ for a suitable constant $c_0=c_0(k) \in [1,2]$.
It follows that $x$ is symmetric; has variance $1$; satisfies that  $\|x\|_{L_4} \leq 2$; and takes the values $\pm{1/c_0}$ and $\pm{k^{1/4}/c_0}$.

Let $x_1,...,x_d$ be independent copies of $x$ and set $X=(x_1,...,x_d)$. Using the independence and symmetry of the coordinates of $X$, it is standard to verify that $X$ is isotropic and satisfies $L_4-L_2$ norm equivalence with an absolute constant $L$.
At the same time, we have the following:

\begin{lemma}
\label{lem:B1d.heavy}
For every $\kappa>0$ there is a constant $c(\kappa)$ such that if $N\leq c d^{1/4}$, then $r_{\mathcal{Q}}^{\rm rad}(\kappa)\geq 1$.
\end{lemma}

\begin{proof}
Let $U=\frac{1}{\sqrt{N}} \sum_{i=1}^N x_i$ and for every $j=1,\dots,d$ set 
\[U_j = \frac{1}{\sqrt{N}} \sum_{i=1}^N \inr{X_i,e_j} = \frac{1}{\sqrt{N}} \sum_{i=1}^N x_{ij}.\]
Thus $(U_j)_{j=1}^d$ are  independent and distributed according to $U$.
	Observe that for every $r\in(0,1]$, 
	$$
(\ast)=
\E \sup_{t \in (B_1^d-B_1^d)\cap r B_2^d} \frac{1}{\sqrt{N}} \sum_{i=1}^N \eps_i \inr{X_i,t} 
\geq   r \E \max_{1 \leq j \leq d} |U_j|,
$$
and that if $\PP(|U| \geq M) \geq \frac{10}{d}$, then $\E \max_{1 \leq j \leq d} |U_j|  \geq c_1M$ for an absolute constant $c_1$.

To identify a suitable $M$, fix a constant $c_2$ and consider the event
$$
{\mathcal A}= \left\{ \text{There is  a single index $1\leq s\leq N$ s.t.}\ |x_s| = \frac{k^{1/4}}{c_0} \ \text{ and } \ \left| \sum_{i=1}^N \eps_i \right| \leq c_2\sqrt{N}\right\}.
$$
On that event,
\begin{equation*}
 \left|\sum_{i=1}^N x_i \right| \geq |x_s| - \left|\sum_{i \neq s} \eps_i \cdot \frac{1}{c_0}\right|
\geq |x_s| - \frac{1}{c_0} \left(\left|\sum_{i=1}^N \eps_i \right|+1\right)
\geq \frac{1}{c_0} (k^{1/4} - c_3 \sqrt{N}).
\end{equation*}
Therefore, if $k^{1/4} \geq 2c_3 \sqrt{N}$ and $\PP({\mathcal A}) \geq \frac{10}{d}$, then $\E \max_{1 \leq j \leq d} |U_j| \geq c_4 \frac{k^{1/4}}{\sqrt{N}}$.

To control $\PP({\mathcal A})$, observe that by independence and for $c_2$ sufficiently large,
$$
\PP({\mathcal A}) \geq \frac{N}{k} \cdot \left(1-\frac{1}{k}\right)^{N-1} \cdot \frac{1}{2} \geq \frac{10}{d},
$$
provided that $N \leq k \leq c_5Nd$.

With that in mind, set  $k=c_5 d$, which is an admissible choice and therefore
\[(\ast) \geq r c_4  \frac{k^{1/4}}{\sqrt N}
=r c_4c_5^{1/4}  \frac{d^{1/4}}{\sqrt N}.\]
In particular, if $N\leq c_6(\kappa)  d^{1/4}$ for a well-chosen constant $c_6(\kappa)$, then $(\ast) > \kappa  r \sqrt N$, and hence  $r^{\rm rad}_{\mathcal{Q}}(\kappa)\geq  r$.
\end{proof}

\end{document}